\documentclass[11pt]{article}
\usepackage{amsfonts}
\usepackage[latin9]{inputenc}
\usepackage{amsmath}
\usepackage{amssymb}
\usepackage{breqn}
\usepackage{url}
\usepackage{graphicx}
\usepackage[pdfstartview=FitH]{hyperref}
\usepackage{amsthm}
\usepackage{hyperref}
\usepackage{url}
\usepackage{enumitem}
\usepackage{fullpage}
\usepackage{verbatim}
\usepackage{bbm}

\newcommand{\supp}{\operatorname{supp}}

\newcommand{\Span}{\operatorname{span}}

\newcommand{\Spec}{\operatorname{Spec}}
\newcommand{\conv}{\operatorname{conv}}
\newcommand{\overlap}{\operatorname{overlap}}
\makeatletter

\begin{document}
\newtheorem{theorem}{Theorem}[section]
\newtheorem{lemma}[theorem]{Lemma}
\newtheorem{definition}[theorem]{Definition}
\newtheorem{claim}[theorem]{Claim}
\newtheorem{example}[theorem]{Example}
\newtheorem{remark}[theorem]{Remark}
\newtheorem{proposition}[theorem]{Proposition}
\newtheorem{corollary}[theorem]{Corollary}
\newtheorem{observation}[theorem]{Observation}
\newcommand{\subscript}[2]{$#1 _ #2$}
\newtheorem*{theorem*}{Theorem}
\author{
Izhar Oppenheim
\footnote{Department of Mathematics, Ben-Gurion University of the Negev, Be'er Sheva 84105, Israel, izharo@bgu.ac.il }
}

\title{Local Spectral Expansion Approach to High Dimensional Expanders Part II: Mixing and Geometrical overlapping}
\maketitle

\begin{abstract}
In this paper, we further explore the local-to-global approach for expansion of simplicial complexes that we call local spectral expansion. 
Specifically, we prove that local expansion in the links imply the global expansion phenomena of mixing and geometric overlapping. Our mixing results also give tighter bounds on the error terms compared to previously known results.
\end{abstract}

\section{Introduction}

This article explores local-to-global expansion behaviour in simplicial complexes. In \cite{LocSpecI}, we showed how local spectral gaps in $1$-dimensional links of a complex implies spectral gaps in the links of of every dimension (including a spectral gap in the $1$-skeleton of the complex). In this article, we show how spectral gaps in the links of a complex imply mixing and topological overlapping, which are global expansion phenomena.  

A pure $n$-dimensional simplicial complex $X$ is a simplicial complex in which every simplex is contained in an $n$-dimensional simplex. The sets with $i+1$ elements are denoted $X(i)$, $0 \leq i \leq n$. The {\em one-skeleton} of the complex $X$ is its underlying graph obtained by $X(0) \cup X(1)$. The {\em link} of $\tau \in X(i)$ denoted $X_{\tau}$ is the complex obtained by taking all faces in $X$ that contain $\tau$ and removing $\tau$ from them. Thus, if $\tau$ is of dimension $i$, then $X_{\tau}$ is of dimension $n-i-1$.

For every $-1 \leq i \leq n-2$, the one skeleton of $X_{\tau}$ is a graph. Below we will consider a weighted random walk on the one skeleton of $X_\tau$ and its spectrum. We do not specify the weights on the links in the introduction in order to ease the reading, and exact definitions are given in the body of the article. We show that bounding the spectra in all the links yield the following mixing results:
 
\begin{theorem}[Mixing for general complexes - informal]
\label{mixing thm for general complexes - intro}
Let $X$ be a pure $n$-dimensional simplicial complex with connected links. Assume that $X$ is $K$-regular in the following sense: for every $\lbrace v \rbrace \in X(0)$, $v$ is contained in exactly $K$ $n$-dimensional simplices of $X$. Assume that there is a constant $0< \lambda < 1$ such that for every $\tau \in \bigcup_{k=-1}^{n-2} X(k)$ the spectrum of the weighted random walk on $X_\tau$ is contained in $[-\lambda, \lambda] \cup \lbrace 1 \rbrace$, then for every pairwise disjoint sets $U_0,...,U_n \subseteq X(0)$,
$$\left\vert \vert X(U_0,...,U_n) \vert - \dfrac{n! K}{\vert X(0) \vert^n} \vert U_0 \vert \vert U_1 \vert ... \vert U_n \vert \right\vert \leq C_n n! \lambda K  \min_{0 \leq i < j \leq n} \sqrt{ \vert U_i \vert \vert U_j \vert},$$
where $\vert X(U_0,...,U_n) \vert$ is the number of $n$-dimensional simplices that have exactly one vertex in each of the sets $U_0,...,U_n$ and $C_n$ is a constant given by the formula:
$$C_n = \sum_{k=0}^{n-1} ((k+1)(k+2)^{n-k} - (k+1)^{n-k+1} ).$$ 
\end{theorem}

An $n$-dimensional simplicial complex is called $(n+1)$-partite if there are sets $S_0,...,S_n \subseteq X(0)$, called the sides of $X$, such that every $n$-dimensional simplex of $X$ has exactly one vertex in each of the sets $S_0,...,S_n$ (this is the high dimensional analogue to a bipartite graph). In the partite case, mixing can be deduced only from one sided spectral gap (as in the case of partite graphs):

\begin{theorem}[Mixing for partite complexes - informal]
\label{mixing thm for partite complexes - intro}
Let $X$ be a pure $n$-dimensional simplicial complex with connected links. Assume that $X$ is $(n+1)$-partite with sides $S_0,...,S_n$. Assume also that $X$ is partite-regular in the following sense: for every $0 \leq i \leq n$ there is a constant $K_i$ such that for every $\lbrace v \rbrace \in S_i$, $v$ is contained in exactly $K_i$ $n$-dimensional simplices of $X$. Assume that there is a constant $0< \lambda < 1$ such that for every $\tau \in \bigcup_{k=-1}^{n-2} X(k)$ the spectrum of the weighted random walk on $X_\tau$ is contained in $[-1, \lambda] \cup \lbrace 1 \rbrace$, then for every sets $U_0 \subseteq S_0,...,U_n \subseteq S_n$,
$$\left\vert \dfrac{\vert X(U_0,...,U_n) \vert}{\vert X(n) \vert} - \dfrac{\vert U_0 \vert ... \vert U_n \vert}{\vert S_0 \vert ... \vert S_n \vert} \right\vert \leq C_n \lambda \min_{0 \leq i < j \leq n} \sqrt{\dfrac{\vert U_i \vert \vert U_j \vert}{\vert S_i \vert \vert S_j \vert}},$$
where 
$$C_n = \sum_{k=0}^{n-1} n! \dfrac{(n+1-k)^{n-k}}{(n-k-1)!}((k+1)(k+2)^{n-k} - (k+1)^{n-k+1} ).$$ 
\end{theorem}

In light of the results of the author in \cite{LocSpecI}, the above mixing results can be deduced based only on the spectral gaps of the $1$-dimensional links (see Corollary \ref{Explicit spec descent coro} in the appendix):
\begin{corollary}
\label{Mixing + spec descent coro - general}
Let $X$ be a pure $n$-dimensional simplicial complex with connected links. Assume that $X$ is $K$-regular in the following sense: for every $\lbrace v \rbrace \in X(0)$, $v$ is contained in exactly $K$ $n$-dimensional simplices of $X$. Assume that there is a constant $0< \lambda < 1$ such that for every $\tau \in X(n-2)$ the spectrum of the simple random walk on $X_\tau$ is contained in $[-\frac{\lambda}{1+(n-1)\lambda}, \frac{\lambda}{1+(n-1)\lambda}] \cup \lbrace 1 \rbrace$, then for every pairwise disjoint sets $U_0,...,U_n \subseteq X(0)$,
$$\left\vert \vert X(U_0,...,U_n) \vert - \dfrac{n! K}{\vert X(0) \vert^n} \vert U_0 \vert \vert U_1 \vert ... \vert U_n \vert \right\vert \leq C_n n! \lambda K  \min_{0 \leq i < j \leq n} \sqrt{ \vert U_i \vert \vert U_j \vert},$$
where $\vert X(U_0,...,U_n) \vert$ is the number of $n$-dimensional simplices which have one vertex in each of the sets $U_0,...,U_n$ and $C_n$ is a constant given by the formula:
$$C_n = \sum_{k=0}^{n-1} ((k+1)(k+2)^{n-k} - (k+1)^{n-k+1} ).$$ 
\end{corollary}

\begin{corollary}
\label{Mixing + spec descent coro - partite}
Let $X$ be a pure $n$-dimensional simplicial complex with connected links. Assume that $X$ is $(n+1)$-partite with sides $S_0,...,S_n$. Assume also that $X$ is partite-regular in the following sense: for every $0 \leq i \leq n$ there is a constant $K_i$ such that for every $\lbrace v \rbrace \in S_i$, $v$ is contained in exactly $K_i$ $n$-dimensional simplices of $X$. Assume that there is a constant $0< \lambda < 1$ such that for every $\tau \in X(n-2)$ the spectrum of the simple random walk on $X_\tau$ is contained in $[-1, \frac{\lambda}{1+(n-1)\lambda}] \cup \lbrace 1 \rbrace$, then for every sets $U_0 \subseteq S_0,...,U_n \subseteq S_n$,
$$\left\vert \dfrac{\vert X(U_0,...,U_n) \vert}{\vert X(n) \vert} - \dfrac{\vert U_0 \vert ... \vert U_n \vert}{\vert S_0 \vert ... \vert S_n \vert} \right\vert \leq C_n \lambda \min_{0 \leq i < j \leq n} \sqrt{\dfrac{\vert U_i \vert \vert U_j \vert}{\vert S_i \vert \vert S_j \vert}},$$
where 
$$C_n = \sum_{k=0}^{n-1} n! \dfrac{(n+1-k)^{n-k}}{(n-k-1)!}((k+1)(k+2)^{n-k} - (k+1)^{n-k+1} ).$$ 
\end{corollary}

\begin{remark}
The results of \cite{LocSpecI}, were stated in terms of the Laplacians and not the spectral gaps of the random walks, therefore we include the formulation of these results in the terminology of random walks in the appendix. 
\end{remark}

It is interesting to compare these mixing result to previously known results: 
\begin{itemize}
\item In the general case, mixing for random complexes was proven by Parzanchevski, Rosenthal and Tessler \cite{PRT}. The error term in that work was of the form $\sqrt{\vert U_0 \vert \vert U_n \vert} \vert U_1 \vert ... \vert U_{n-1} \vert$. Also, a criterion for mixing for general (non-random) complexes based on the spectra of all high dimensional Laplacians was proven by Parzanchevski \cite{Par} and the error term in that result was $\max_i \vert U_i \vert$. Our results improve the error terms in both works. 
\item  Mixing for partite Ramanujan complexes was proven by Evra, Golubev and Lubotzky \cite{EGL}. The technique of \cite{EGL} is very different from ours and relied on a quantitative version of property (T). The error term in \cite{EGL} is independent of the size of the sets, i.e., in \cite{EGL} the bound on $\left\vert \frac{\vert X(U_0,...,U_n) \vert}{\vert X(n) \vert} - \frac{\vert U_0 \vert ... \vert U_n \vert}{\vert S_0 \vert ... \vert S_n \vert} \right\vert$ is indpenedent on $\vert U_0 \vert,...,\vert U_n \vert$ and depends only on the thickness (or equivalently, the local spectral gap) of the complex. Our results improve on the work of \cite{EGL} regarding both generality (since it is purely combinatorial) and the error term.  
\end{itemize}

Using mixing, we derive geometric overlapping property (see below). Our work provides new examples of bounded degree families of complexes with (uniform) geometric overlapping, because it implies geometric overlapping for 
\begin{enumerate}
\item Quotients of affine buildings of any type (assuming large enough thickness), while previous results \cite{FGLNP} dealt only with quotients of $\widetilde{A}$ buildings.
\item The new constructions of high dimensional expander of the author and Tali Kaufman \cite{KO-construction}.   
\end{enumerate}

\begin{remark}
Earlier drafts of these results appeared online, but the proofs were given using Laplacians and not random walks operators. We found that passing to the language of random walks simplify the proofs considerably, since there is not need to deal with issues of orientation.
\end{remark}

\paragraph{Structure of the paper.} In Section \ref{Weighted simplicial complexes sec}, we recall the basic terminology of weighted simplicial complexes and further develop this framework in for simplices defined by vertex sets and partite simplicial complexes. In Section \ref{The signless differential and random walks sec}, we define the ``signless differential'' and show the connection between this differential and random walks operators. In Section \ref{Non-trivial spectrum and random walks in the partite case sec}, we develop some results regarding the spectrum of random walks in partite complex that are needed to prove mixing in the partite case. In Section \ref{Links and localization sec}, we develop localization results of the random walks operator \`{a} la Garland in the case of the signless differential. In Section \ref{Random walks on sets of simplices determined by vertex sets sec}, we study the connections between the upper and lower random walks on sets of simplices determined by vertex sets. In Section \ref{Mixing sec}, we prove the main mixing results of this article. In Section  \ref{Geometric overlapping sec}, we show how to deduce geometric overlapping from our mixing results, based on a result of Pach. In the appendix, we state the spectral descent results of \cite{LocSpecI} in terms of random walks.

\begin{remark}
Part of the terminology and results of Sections \ref{The signless differential and random walks sec} and \ref{Links and localization sec} already appeared in an article of the author with Tali Kaufman \cite{KO-RW}.
\end{remark}

\section{Weighted simplicial complexes}
\label{Weighted simplicial complexes sec}
Let $X$ be a pure $n$-dimensional finite simplicial complex. For $-1 \leq k \leq n$, we denote $X (k)$ to be the set of all $k$-simplices in $X$ ($X (-1) = \lbrace \emptyset \rbrace$). A balanced weight function $m$ on $X$ is a function:
$$m : \bigcup_{-1 \leq k \leq n} X (k) \rightarrow \mathbb{R}^+,$$
such that for every $-1 \leq k \leq n-1$ and for every $\tau \in X (k)$ we have that
$$m (\tau) = \sum_{\sigma \in X (k+1)} m(\sigma).$$
By its definition, it is clear the $m$ is determined by the values in takes in $X (n)$. A simplicial complex with a balanced weight function will be called a weighted simplicial complex. For a pure $n$-dimensional simplicial complex there is a natural balanced weight function $m_h$ which we call the homogeneous weight function defined by $m_h (\sigma) = 1$ for every $\sigma \in X(n)$.

\begin{proposition}\cite{LocSpecI}[Proposition 2.7]
\label{weight in n dim simplices}
For every $-1 \leq k \leq n$ and every $\tau \in X (k) $ we have that
$$\dfrac{1}{(n-k)!}  m (\tau) =\sum_{\sigma \in X (n), \tau \subseteq \sigma} m (\sigma ),$$
where $\tau \subseteq \sigma$ means that $\tau$ is a face of $\sigma$.

In particular, the homogeneous weight $m$ on $X$ can be written explicitly as
$$ \forall -1 \leq k \leq n, \forall \tau \in X (k), \dfrac{1}{(n-k)!}  m(\tau) =  \vert \lbrace \sigma \in X (n) : \tau \subseteq \sigma \rbrace \vert .$$

\end{proposition}

\begin{corollary}\cite{LocSpecI}[Corollary 2.8]
\label{weight in l dim simplices}
For every $-1 \leq k < l \leq n$ and every $\tau \in X (k) $ we have
$$\dfrac{1}{(l-k)!} m(\tau) = \sum_{\sigma \in X (l), \tau \subset \sigma} m(\sigma) .$$
\end{corollary}

For $-1 \leq k \leq n$ and a set $\emptyset \neq U \subseteq X (k)$, we denote
$$m(U) = \sum_{\sigma \in U} m(\sigma ).$$

\begin{proposition}
\label{weight of X^k compared to X^l}
For every $-1 \leq k < l \leq n$,
$$m(X (k)) = \dfrac{(l+1)!}{(k+1)!} m (X (l)). $$
\end{proposition}

\begin{proof}
By corollary \ref{weight in l dim simplices}, we have that
\begin{dmath*}
m(X (k)) = \sum_{\tau \in X (k)} m(\tau ) = \sum_{\tau \in X (k)} (l-k)! \sum_{\sigma \in X (l), \tau \subset \sigma} m(\sigma) =\sum_{\sigma \in X (l)} (l-k)! m(\sigma) \sum_{\tau \in X (k), \tau \subset \sigma} 1 = \sum_{\sigma \in X (l)} (l-k)! m(\sigma) {l+1 \choose k+1} = \dfrac{(l+1)!}{(k+1)!} m(X (l)) .
\end{dmath*}
\end{proof}

Throughout this article, $X$ is a pure $n$-dimensional finite weighted simplicial complex with a balanced weight function $m$. 

\subsection{Sets of simplices determined by vertex sets}
Given a weight simplicial complex $X$ as above, $0 \leq k \leq n$ and $U_0,...,U_k \subseteq X(0)$ pairwise disjoint sets of vertices of $X$, we define $X(U_0,...,U_k) \subseteq X(k)$ as follows: 
$$X(U_0,...,U_k) = \lbrace \sigma \in X(k) : \forall 0 \leq i \leq k, \vert \sigma \cap U_i \vert =1 \rbrace,$$
i.e., $X(U_0,...,U_k)$ is the set of $k$-simplices of $X$ that have one vertex in each of the $U_i$'s. Note that by this definition $X(U_i) = U_i$.

\subsection{Weighted partite simplicial complexes}

A pure $n$-dimensional simplicial complex $X$ is called $(n+1)$-partite if the are disjoint sets $S_0,...,S_n \subseteq X(0)$ called the sides of $X$ such that $X(0) =\bigcup_{i=0}^n S_i$ and such that for every $\sigma \in X(n)$ and every $0 \leq i \leq n$, $\vert \sigma \cap S_i \vert =1$, i.e., all the $n$-dimensional simplices in $X$ have one vertex in each side of $X$.

We will need the following facts regarding weighted partite complexes.
\begin{proposition}
\label{weight in partite complex prop}
Let $X$ be a pure $n$-dimensional, $(n+1)$-partite simplicial complex with sides $S_0,...,S_n$. Then for every $-1 \leq k \leq n-1$, every $\tau \in X(k)$ and every $0 \leq i \leq n$, the following holds:
$$\sum_{\sigma \in X(k+1), \tau \subseteq \sigma, \sigma \cap S_i \neq \emptyset} m(\sigma) = \begin{cases}
m(\tau) & \tau \cap S_i \neq \emptyset \\
\frac{1}{n-k} m(\tau) & \tau \cap S_i = \emptyset
\end{cases}.$$
\end{proposition}

\begin{proof}
We can assume without loss of generality that $i=0$. In the case where $\tau \cap S_0 \neq \emptyset$, then for every $\sigma \in X(k+1)$ such that $\tau \subseteq \sigma$ it holds that $\sigma \cap S_0 \neq \emptyset$ and the equality follows by the definition of the balanced weight function. Assume that $\tau \cap S_0 = \emptyset$, then for every $\eta \in X(n)$ such that $\tau \subseteq \eta$, there is a unique $\sigma \in X(k+1)$ such that $\tau \subseteq \sigma \subseteq \eta$ and $\sigma \cap S_0 \neq \emptyset$. Therefore by Proposition \ref{weight in n dim simplices}, 
\begin{dmath*}
\sum_{\sigma \in X(k+1), \tau \subseteq \sigma, \sigma \cap S_0 \neq \emptyset} m(\sigma) =
\sum_{\sigma \in X(k+1), \tau \subseteq \sigma, \sigma \cap S_0 \neq \emptyset} (n-(k+1))! \sum_{\eta \in X(n), \sigma \subseteq \eta} m(\eta) =
(n-(k+1))! \sum_{\eta \in X(n), \tau \subseteq \eta} m(\eta) = \frac{1}{n-k} m(\tau),
\end{dmath*}
where the last equality is again by Proposition \ref{weight in n dim simplices}.
\end{proof}

\begin{corollary}
\label{weight of S_i coro}
Let $X$ be a pure $n$-dimensional, $(n+1)$-partite simplicial complex with sides $S_0,...,S_n$. Then for every $0 \leq i \leq n$, $m(S_i) = \frac{1}{n+1} m(X(0))$.
\end{corollary}

\begin{proof}
Apply the above proposition in the case $k=-1$. 
\end{proof}

\section{The signless differential and random walks}
\label{The signless differential and random walks sec}
For $-1 \leq k \leq n-1$, we denote $C^k (X, \mathbb{R})$ to be the set of all functions $\phi : X (k) \rightarrow \mathbb{R}$. Abusing the terminology, we will call the space $C^k (X, \mathbb{R})$ the space of non-oriented cochains. On $C^k (X,\mathbb{R})$ define the following inner-product:
$$\forall \phi, \psi \in C^k (X,\mathbb{R}), \langle \phi, \psi \rangle = \sum_{\sigma \in  X (k)} m (\sigma ) \phi (\sigma) \psi (\sigma).$$
Denote by $\Vert . \Vert$ the norm induced by this inner-product.

\begin{definition}
For $-1 \leq k \leq n-1$, the signless $k$-differential is an operator $d_k : C^k (X,\mathbb{R}) \rightarrow C^{k+1} (X,\mathbb{R})$ defined as:
$$\forall \phi \in C^k (X,\mathbb{R}), \forall \sigma \in X (k+1), d_k \phi (\sigma) = \sum_{\tau \subset \sigma, \tau \in X (k)} \phi (\tau).$$
Define $(d_{k})^* : C^{k+1} (X,\mathbb{R}) \rightarrow  C^{k} (X,\mathbb{R})$ to be the adjoint operator to $d_k$, i.e., the operator such that for every $\phi \in  C^{k} (X,\mathbb{R}), \psi \in C^{k+1} (X,\mathbb{R})$, $\langle d_k \phi, \psi \rangle = \langle \phi, (d_{k})^* \psi \rangle$.
\end{definition}

\begin{remark}
We note that the signless differential is not a differential in the usual sense, since $d_{k+1} d_k \neq 0$. The name signless differential stems from the fact that this is the operator we will use in lieu of the differential in our setting (note that since our non-oriented cochains are defined without using orientation of simplices, we cannot use the usual differential).
\end{remark}

Below, we will sometimes omit the index of signless differential and its adjoint and just denote $d, d^*$ where $k$ will be implicit.

\begin{lemma}
\label{d* lemma}
For $-1 \leq k \leq n-1$, $d^* : C^{k+1} (X,\mathbb{R}) \rightarrow  C^{k} (X,\mathbb{R})$ is the operator
$$\forall \psi \in C^{k+1} (X,\mathbb{R}), \forall \tau \in X (k), d^* \psi (\tau) = \sum_{\sigma \in X (k+1), \tau \subset \sigma} \dfrac{m(\sigma)}{m(\tau)} \psi (\sigma).$$
\end{lemma}

\begin{proof}
Let $\phi \in C^k (X,\mathbb{R})$ and $\psi \in C^{k+1} (X,\mathbb{R})$. Then
\begin{dmath*}
\langle d \phi, \psi \rangle = \sum_{\sigma \in X (k+1)} m (\sigma) d \phi (\sigma) \psi (\sigma) = \sum_{\sigma \in X (k+1)} m (\sigma) \sum_{\tau \in X (k), \tau \subset \sigma} \phi (\tau) \psi (\sigma) = \sum_{\tau \in X (k)} \phi (\tau) \sum_{\sigma \in X (k+1), \tau \subset \sigma} m (\sigma)  \psi (\sigma) = \sum_{\tau \in X (k)} m(\tau) \phi (\tau) \left( \sum_{\sigma \in X (k+1), \tau \subset \sigma} \dfrac{m(\sigma)}{m(\tau)}  \psi (\sigma) \right) = \langle \phi, d^* \psi \rangle.
\end{dmath*}
\end{proof}

For $X$ as above we will define the following random walks on simplices of $X$:

\begin{definition}
For $0 \leq k \leq n-1$, the upper random walk on $k$-simplices is defined by the transition probability matrix $M^+_k : X (k) \times X (k) \rightarrow \mathbb{R}$:
$$M^+_k (\tau, \tau') = \begin{cases}
\frac{1}{k+2} & \tau = \tau' \\
\frac{m (\tau \cup \tau')}{(k+2) m(\tau)} & \tau \cup \tau' \in X (k+1) \\
0 & \text{otherwise}
\end{cases}.$$
\end{definition}

\begin{definition}
For $0 \leq k \leq n$, the lower random walk on $k$-simplices is defined by the transition probability matrix $M^-_k : X (k) \times X (k) \rightarrow \mathbb{R}$:
$$M^-_k (\tau, \tau') = \begin{cases}
\sum_{\eta \in X (k-1)} \frac{m(\tau)}{(k+1) m (\eta)} & \tau = \tau' \\
\frac{m (\tau')}{(k+1) m(\tau \cap \tau')} & \tau \cap \tau' \in X (k-1) \\
0 & \text{otherwise}
\end{cases}.$$
\end{definition}

We leave it to the reader to check that those are in fact transition probability matrix, i.e., that for every $\tau$, $\sum_{\tau'} M^\pm_k (\tau,\tau') =1$. We note that both the random walks defined above are lazy in the sense that $M^\pm (\tau, \tau) \neq 0$. 

In the case of the upper random walk, one can easily define a non lazy random walk as follows:
\begin{definition}
\label{non lazy walk def}
For $0 \leq k \leq n-1$, the non-lazy upper random walk on $k$-simplices is defined by the transition probability matrix $(M')^+_{k} : X (k) \times X (k) \rightarrow \mathbb{R}$:
$$(M')_{k}^+ = \frac{k+2}{k+1} \left( M^+_k - \frac{1}{k+2} I \right) = \frac{k+2}{k+1} M^+_k - \frac{1}{k+1} I .$$
\end{definition}

It is standard to view $M^{\pm}_k, (M')^+_{k}$ as averaging operators on $C^k (X,\mathbb{R})$ and we will not make the distinction between the transition probability matrix and the averaging operator it induces.

\begin{remark}
\label{M^-_0, M'^+_0 remark}
It is worth noting that $M^{-}_0$ and $(M')^+_{0}$ are familiar operators/matrices: $M^-_0$ is  a projection on the space of the constant functions (on vertices) with respect to the inner-product defined above, and $(M')^+_{0}$ is the weighted (non-lazy) random walk matrix of the $1$-skeleton of $X$, i.e., $\Delta^+_0 = I-(M')^+_{0}$ is the weighted normalized Laplacian on the $1$-skeleton (which was used in \cite{LocSpecI}).
\end{remark}

\begin{lemma}
\label{connection between d*d and M lemma}
For $0 \leq k \leq n-1$ and $\phi \in C^k (X,\mathbb{R})$, $d^*_k d_k \phi = (k+2) M^{+} \phi$ and $d_{k-1} d_{k-1}^* \phi= (k+1) M^{-} \phi$.
\end{lemma}

\begin{proof}
Let $\phi \in C^k (X,\mathbb{R})$ and $\tau \in X (k)$, then
\begin{dmath*}
d^* d \phi (\tau) = \sum_{\sigma \in X (k+1), \tau \subset \sigma} \dfrac{m(\sigma)}{m(\tau)} d \phi (\sigma) =
\sum_{\sigma \in X (k+1), \tau \subset \sigma} \dfrac{m(\sigma)}{m(\tau)} \sum_{\tau' \in X (k), \tau' \subset \sigma} \phi (\tau') =
\sum_{\sigma \in X (k+1), \tau \subset \sigma} \dfrac{m(\sigma)}{m(\tau)} \sum_{\tau' \in X (k), \tau' \subset \sigma, \tau' \neq \tau} \phi (\tau') + \sum_{\sigma \in X (k+1), \tau \subset \sigma} \dfrac{m(\sigma)}{m(\tau)} \phi (\tau).
\end{dmath*}
Note that
$$\sum_{\sigma \in X (k+1), \tau \subset \sigma} \dfrac{m(\sigma)}{m(\tau)} \phi (\tau) = \dfrac{m(\tau)}{m(\tau)} \phi (\tau) = \phi (\tau).$$
Also note that
\begin{dmath*}
\sum_{\sigma \in X (k+1), \tau \subset \sigma} \dfrac{m(\sigma)}{m(\tau)} \sum_{\tau' \in X (k), \tau' \subset \sigma, \tau' \neq \tau} \phi (\tau')  =
\sum_{\sigma \in X (k+1), \tau \subset \sigma} \sum_{\tau' \in X (k), \tau' \subset \sigma, \tau' \neq \tau} \dfrac{m(\tau' \cup \tau)}{m(\tau)} \phi (\tau') =
\sum_{\tau' \in X (k), \tau \cup \tau' \in X (k+1)} \dfrac{m(\tau \cup \tau')}{m(\tau)} \phi (\tau').
\end{dmath*}
Therefore
$$d^* d \phi (\tau) = \phi (\tau) + \sum_{\tau' \in X (k), \tau \cup \tau' \in X (k+1)} \dfrac{m(\tau \cup \tau')}{m(\tau)} \phi (\tau') = (k+2) M^+ \phi (\tau).$$
Similarly,
\begin{dmath*}
d d^* \phi (\tau) = \sum_{\eta \in X (k-1), \eta \subset \tau} d^* \phi (\eta) =
\sum_{\eta \in X (k-1), \eta \subset \tau} \sum_{\tau' \in X (k), \eta \subset \tau'} \dfrac{m(\tau')}{m(\eta)} \phi (\tau') =
\sum_{\eta \in X (k-1), \eta \subset \tau} \sum_{\tau' \in X (k), \tau' \neq \tau, \eta \subset \tau'} \dfrac{m(\tau')}{m(\eta)} \phi (\tau') + \sum_{\eta \in X (k-1), \eta \subset \tau} \dfrac{m(\tau)}{m(\eta)} \phi (\tau) =
\sum_{\eta \in X (k-1), \eta \subset \tau} \sum_{\tau' \in X (k), \tau' \cap \tau = \eta} \dfrac{m(\tau')}{m(\tau \cap \tau')} \phi (\tau') + \sum_{\eta \in X (k-1), \eta \subset \tau} \dfrac{m(\tau)}{m(\eta)} \phi (\tau) =
\sum_{\tau' \in X (k), \tau \cap \tau' \in X (k-1)} \dfrac{m(\tau')}{m(\tau \cap \tau')} \phi (\tau') + \sum_{\eta \in X (k-1), \eta \subset \tau} \dfrac{m(\tau)}{m(\eta)} \phi (\tau) = (k+1) M^- \phi (\tau).
\end{dmath*}
\end{proof}

\begin{corollary}
\label{norm of d^* d, d d^* coro}
For every $0 \leq k \leq n-1$, $\Vert d_k^* d_k \Vert = k+2$ and $\Vert d_{k-1} d_{k-1}^* \Vert = k+1$.
\end{corollary}

\section{Non-trivial spectrum and random walks in the partite case}
\label{Non-trivial spectrum and random walks in the partite case sec}

In \cite{LocSpecI}[Section 5.2], we analyzed the specturum of the Laplacian on the $1$-skeleton of an $n$-dimensional $(n+1)$-partite complex. Below, we recall the definitions and results of \cite{LocSpecI}[Section 5.2] using the terminology of random walks (instead of Laplacians - see Remark \ref{M^-_0, M'^+_0 remark}). First, we note that by \cite{LocSpecI}[Proposition 5.2], 
$$\varphi_i (\lbrace v \rbrace) = \begin{cases}
n & v \in S_i \\
-1 & \text{otherwise}
\end{cases}.$$
is an eigenfunction of $(M')^{+}_{0}$ with eigenvalue $-\frac{1}{n}$. Also, as in every graph, the constant function $\mathbbm{1}$ is an eigenfunction with the eigenvalue $1$. Therefore, as in \cite{LocSpecI}, we denote $C_{nt}^0 (X, \mathbb{R}) = \Span  \lbrace \mathbbm{1}, \varphi_0,..., \varphi_n \rbrace^\perp$ to be the subspace of $C^0 (X, \mathbb{R})$ of the non-trivial eigenfunctions of $(M')^{+}_{0}$. By \cite{LocSpecI}[Proposition 5.3],  $C_{nt}^0 (X, \mathbb{R}) = \Span  \lbrace \chi_{S_0},..., \chi_{S_n} \rbrace^\perp$, where $\chi_{S_i}$ is the indicator function on $S_i$. Denote
$M^{-,p}_0$ to be the orthogonal projection on $\Span  \lbrace \chi_{S_0},..., \chi_{S_n} \rbrace$, explicitly, for every $\phi \in C^0 (X, \mathbb{R})$, 
$$M^{-,p}_0 \phi =  \sum_{i=0}^n \left( \sum_{v \in S_i}  \dfrac{m(\lbrace v \rbrace)}{m (S_i)} \phi (\lbrace v \rbrace) \right) \chi_{S_i}.$$
Then $M^{-,p}_0$ commutes with $(M')^{+}_{0}$ (because it is a projection on a subspace spanned by eigenfunctions of $(M')^{+}_{0}$) and $I-M^{-,p}_0$ is the orthogonal projection on $C_{nt}^0 (X, \mathbb{R})$. Last, we recall that the spectrum of $(M')^{+}_{0}$ in $C_{nt}^0 (X, \mathbb{R})$ has the following symmetry:
\begin{proposition}\cite{LocSpecI}[Lemma 5.5]
\label{symmetry of spectrum in partite case - prop}
Let $X$ be a pure $n$-dimensional, $(n+1)$-partite simplicial complex such that the one-skeleton of $X$ is connected. Denote
$$\lambda (X) = \max \lbrace \lambda : \exists \phi \in C_{nt}^0 (X, \mathbb{R}), (M')^{+}_{0} \phi = \lambda \phi \rbrace,$$
 $$\kappa_{partite} (X) = \min \lbrace \kappa : \exists \phi \in C_{nt}^0 (X, \mathbb{R}), (M')^{+}_{0} \phi = \kappa \phi \rbrace,$$
then $- n \lambda (X) \leq \kappa_{partite} (X) \leq -\frac{1}{n} \lambda (X)$.
\end{proposition}

Last, we will need the following additional lemmata:
\begin{lemma}
Let $X$ be a pure $n$-dimensional, $(n+1)$-partite simplicial complex with sides $S_0,...,S_n$. Let $0 \leq i \leq n$ and $\phi \in C^0 (X, \mathbb{R})$ such that $\supp (\phi) \subseteq S_i$, then
$$M^{-,p}_0 \phi (\lbrace v \rbrace) = \begin{cases}
0 & v \notin S_0 \\
(n+1) M_0^{-} \phi & v \in S_0
\end{cases}.$$
\end{lemma}

\begin{proof}
Without loss of generality, we can assume $i=0$ and $\supp (\phi) \subseteq S_0$. If $v \in S_{i_0}$, $i_0 \neq 0$, then
$$M^{-,p}_0 \phi (\lbrace v \rbrace) =    \sum_{u \in S_{i_0}}  \dfrac{m(\lbrace u \rbrace)}{m (S_i)} \phi (\lbrace u \rbrace) = 0.$$
If $v \in S_0$, recall that by Corollary \ref{weight of S_i coro}, $m(S_0) = \frac{1}{n+1} m(X(0))$. Therefore 
\begin{dmath*}
M^{-,p}_0 \phi (\lbrace v \rbrace) =  
\sum_{u \in S_{0}}  \dfrac{m(\lbrace u \rbrace)}{m (S_0)} \phi (\lbrace u \rbrace) = 
(n+1)\sum_{u \in S_{0}}  \dfrac{m(\lbrace u \rbrace)}{m (X(0))} \phi (\lbrace u \rbrace) =
(n+1)\sum_{u \in X(0)}  \dfrac{m(\lbrace u \rbrace)}{m (X(0))} \phi (\lbrace u \rbrace) =  
(n+1) M_0^{-} \phi,
\end{dmath*}
as needed. 
\end{proof}

\begin{lemma}
\label{Lemma M' on M^-,p}
Let $X$ be a pure $n$-dimensional, $(n+1)$-partite simplicial complex with sides $S_0,...,S_n$. Let $0 \leq i \leq n$ and $\phi \in C^0 (X, \mathbb{R})$ such that $\supp (\phi) \subseteq S_i$, then
$$(M')_{0}^+ M^{-,p}_0 \phi (\lbrace v \rbrace) = \begin{cases}
0 & v \in S_i \\
\frac{n+1}{n} M_0^{-} \phi & v \notin S_i 
\end{cases}.$$ 
\end{lemma}

\begin{proof}
If $v \in S_0$, then for every $u \in X(0)$, if $\lbrace u,v \rbrace \in X(1)$, then $u \in S_1 \cup S_2 \cup ... \cup S_n$ and by the previous lemma, $M^{-,p}_0 \phi (\lbrace u \rbrace) =0$. It follows that $(M')_{0}^+ M^{-,p}_0 \phi (\lbrace v \rbrace) = 0$. 
 
If $v \notin S_0$, then, by the previous lemma,
\begin{dmath*}
(M')_{0}^+ M^{-,p}_0 \phi (\lbrace v \rbrace) = 
\sum_{u \in X(0), \lbrace u, v \rbrace \in X(1)} \dfrac{m(\lbrace u, v \rbrace)}{m(\lbrace v \rbrace)}   M^{-,p}_0 \phi (\lbrace u \rbrace) = 
\sum_{u \in S_0, \lbrace u, v \rbrace \in X(1)} \dfrac{m(\lbrace u, v \rbrace)}{ m(\lbrace v \rbrace)}   ((n+1) M^{-}_0 \phi)  = 
(n+1)(M^{-}_0 \phi) \sum_{u \in S_0, \lbrace u, v \rbrace \in X(1)} \dfrac{m(\lbrace u, v \rbrace)}{m(\lbrace v \rbrace)} = \dfrac{n+1}{n} (M^{-}_0 \phi),
\end{dmath*}
where the last equality is due to Proposition \ref{weight in partite complex prop}.

\end{proof}

\section{Links and localization}
\label{Links and localization sec}

This section discusses the idea of links and localization and Garland's method in our setting. Although this is well known (and was already discussed in \cite{LocSpecI}), we are working in a slightly different setting (e.g., random walks instead of Laplacians) and need slightly more general results and therefore we state and prove all the results below.

Let $X$ be a pure $n$-dimensional finite simplicial complex with a weight function $m$. Recall that for $-1 \leq k \leq n-1$, $\tau \in X (k)$, the link of $\tau$, denoted $X_\tau$, is a pure $(n-k-1)$-simplicial complex defined as:
$$\eta \in X_\tau (l) \Leftrightarrow \eta \in X (l) \text{ and } \tau \cup \eta \in X (k+l+1).$$

For $\tau \in X(k)$, $-1 \leq k \leq n-2$, we will say that $X_\tau$ is connected if the one-skeleton of $X$ is connected. Below, we will always assume that for every  $-1 \leq k \leq n-2$ and every $\tau \in X(k)$, $X_\tau$ is connected (note that this also implies that the one-skeleton of $X$ itself is connected).

On $X_\tau$ we define the weight function $m_\tau$ induced by $m$ as $m_{\tau} (\eta) = m (\tau \cup \eta)$.
Using this weight function the inner-product and the norm on $C^l (X_\tau, \mathbb{R})$ are defined as above. The operators $M^\pm_{\tau,l}, (M')^+_{\tau,l}$ and $d_\tau, d^*_\tau$ are also defined on $C^l (X_\tau, \mathbb{R})$ as above.

Given a cochain $\phi \in C^l (X, \mathbb{R})$ and a simplex $\tau \in X (k)$ with $-1 \leq k < l$, we define the localization of $\phi$ on $X_\tau$, denoted $\phi_\tau$ as a cochain $\phi_\tau \in C^{l-k-1} (X_\tau, \mathbb{R})$ defined as
$$\phi_\tau (\eta) = \phi (\tau \cup \eta).$$

The key observation called Gralnad's method (which was initially due to Garland \cite{Garland}, but is now considered standard - see \cite{BS}, \cite{GW}) is that inner-products of cochains and their differentials can be computed via their localizations. In our setting, the Garland's method results we will use below are summarized in the following proposition:
\begin{proposition}
\label{localization proposition}
Let $0 \leq k \leq n$ and let $\phi, \psi \in C^k (X,\mathbb{R})$, then
\begin{itemize}
\item $(k+1) \langle \phi, \psi \rangle = \sum_{\tau \in X (k-1)} \langle \phi_\tau , \psi_\tau \rangle$.
\item For $0 \leq k$,
$\langle d^* \phi, d^* \psi \rangle = \sum_{\tau \in X (k-1)} \langle d^*_\tau \phi_\tau , d^*_\tau \psi_\tau \rangle .$
\item For $0 \leq k<n$, 
$\langle d \phi, d \psi \rangle = \sum_{\tau \in X (k-1)} \left( \langle d_\tau \phi_\tau , d_\tau \psi_\tau \rangle - \dfrac{k}{k+1} \langle \phi_\tau , \psi_\tau \rangle \right).$

\end{itemize}

\end{proposition}

\begin{proof}
Let $\phi, \psi \in C^k (X,\mathbb{R})$, then
\begin{dmath*}
\sum_{\tau \in X (k-1)} \langle \phi_\tau , \psi_\tau \rangle =
\sum_{\tau \in X (k-1)} \sum_{\lbrace v \rbrace \in X_\tau (0)} m_\tau (\lbrace v \rbrace) \phi_\tau (\lbrace v \rbrace) \psi_\tau (\lbrace v \rbrace) = \\
\sum_{\tau \in X (k-1)} \sum_{\lbrace v \rbrace \in X_\tau (0)} m (\tau \cup \lbrace v \rbrace) \phi (\tau \cup \lbrace v \rbrace) \psi (\tau \cup \lbrace v \rbrace) =
\sum_{\tau \in X (k-1)} \sum_{\sigma \in X (k), \tau \subset \sigma} m (\sigma) \phi (\sigma) \psi (\sigma) = \\
\sum_{\sigma \in X (k)} \sum_{\tau \in X (k-1), \tau \subset \sigma} m (\sigma) \phi (\sigma) \psi (\sigma) =
(k+1) \sum_{\sigma \in X (k)} m (\sigma) \phi (\sigma) \psi (\sigma) =
(k+1) \langle \phi , \psi \rangle.
\end{dmath*}
In order to prove the second equality, we notice that for $0 \leq k$ and every $\tau \in X (k-1)$, it holds by definition that $X_\tau (-1) = \lbrace \emptyset \rbrace$. Therefore 
\begin{dmath*}
d^* \phi (\tau) = \sum_{\sigma \in X (k), \tau \subset \sigma} \dfrac{m(\sigma)}{m(\tau)} \phi (\sigma) =
\sum_{\lbrace v \rbrace \in X_\tau (0)} \dfrac{m_\tau (\lbrace v \rbrace)}{m_\tau (\emptyset)} \phi_\tau (\lbrace v \rbrace)
 = d_\tau^* \phi_\tau (\emptyset).
\end{dmath*}
This yields that 
\begin{dmath*}
\langle d^* \phi, d^* \psi \rangle = \sum_{\tau \in X(k-1)} m(\tau) (d^* \phi (\tau)) (d^* \psi (\tau)) = \sum_{\tau \in X(k-1)} m_\tau (\emptyset) (d_\tau^* \phi_\tau (\emptyset)) (d_\tau^* \psi_\tau (\emptyset)) = \sum_{\tau \in X(k-1)} \langle d^*_\tau \phi_\tau , d^*_\tau \psi_\tau \rangle,
\end{dmath*}
as needed.

Last, Assume that $0 \leq k <n$, then for every $\sigma \in X (k+1)$, the following holds:
\begin{dmath*}
(d \phi (\sigma)) (d \psi (\sigma)) =
(\sum_{\eta \in X (k), \eta \subset \sigma} \phi (\eta))  (\sum_{\eta \in X (k), \eta \subset \sigma} \psi (\eta)) =
\sum_{\eta \in X (k), \eta \subset \sigma} \phi (\eta) \psi (\eta) +  \sum_{\eta, \eta' \in X (k), \eta \neq \eta', \eta, \eta' \subset \sigma} 2 \phi (\eta) \psi (\eta') =
\sum_{\eta, \eta' \in X (k), \eta \neq \eta', \eta, \eta' \subset \sigma} (\phi (\eta) + \phi (\eta'))(\psi (\eta) + \psi (\eta')) - k \sum_{\eta \in X (k), \eta \subset \sigma} \phi (\eta) \psi (\eta) =
\sum_{\tau \in X (k-1), \tau \subset \sigma} (d_\tau \phi_\tau (\sigma \setminus \tau)) (d_\tau \psi_\tau (\sigma \setminus \tau))  - k \sum_{\eta \in X (k), \eta \subset \sigma} \phi (\eta) \psi (\eta).
\end{dmath*}
Therefore
\begin{dmath*}
\langle d \phi , d \psi \rangle =
\sum_{\sigma \in X (k+1)} m (\sigma) (d \phi (\sigma)) (d \psi (\sigma)) =
\sum_{\tau \in X (k-1), \tau \subset \sigma} m (\sigma ) (d_\tau \phi_\tau (\sigma \setminus \tau)) (d_\tau \psi_\tau (\sigma \setminus \tau))  - k \sum_{\eta \in X (k), \eta \subset \sigma} m (\sigma ) \phi (\eta) \psi (\eta) =
\sum_{\tau \in X (k-1)} \sum_{\sigma \in X (k+1), \tau \subset \sigma} m (\sigma) (d_\tau \phi_\tau (\sigma \setminus \tau)) (d_\tau \psi_\tau (\sigma \setminus \tau)) - k \sum_{\eta \in X (k)} \phi (\eta) \psi (\eta)  \sum_{\sigma \in X (k+1), \eta \subset \sigma} m (\sigma) =
\sum_{\tau \in X (k-1)} \sum_{\gamma \in X_\tau (k)} m_\tau (\gamma) (d_\tau \phi_\tau (\gamma))(d_\tau \psi_\tau (\gamma)) - k \sum_{\eta \in X (k)} m(\eta) \phi (\eta) \psi (\eta) =
\left( \sum_{\tau \in X (k-1)} \langle d_\tau \phi_\tau , d_\tau \psi_\tau \rangle \right) - k \langle \phi , \psi \rangle =
\sum_{\tau \in X (k-1)} \left( \Vert d_\tau \phi_\tau \Vert^2 - \dfrac{k}{k+1} \langle \phi_\tau , \psi_\tau \rangle \right),
\end{dmath*}
where the last equality is due to the equality
$$\langle \phi , \psi \rangle = \dfrac{1}{k+1} \sum_{\tau \in X (k-1)} \langle \phi_\tau , \psi_\tau \rangle,$$
proven above.
\end{proof}

As a result of Proposition \ref{localization proposition} we deduce the following:

\begin{proposition}
\label{M^+ - M^- proposition}
Let $0 \leq k  \leq n-1$ and let $\phi, \psi \in C^k (X,\mathbb{R})$, then
$$ \langle (d^* d - d d^*) \phi, \psi \rangle  = \langle \phi, \psi \rangle + \sum_{\tau \in X (k-1)}    \langle (M ')^+_{\tau, 0} (I-M^-_{\tau, 0}) \phi_\tau , \psi_\tau \rangle.$$
\end{proposition}

\begin{proof}
Let $\phi \in C^k (X,\mathbb{R})$. Note that for every $\tau \in X (k-1)$, $M^-_{\tau, 0}$ is the orthogonal projection on the space of constant functions in $C^0 (X_\tau, \mathbb{R})$ and therefore $(M ')^+_{\tau, 0} M^-_{\tau, 0} = M^-_{\tau, 0}$.

Further note that by Lemma \ref{connection between d*d and M lemma}
\begin{dmath*}
\langle d_\tau \phi_\tau d_\tau \psi_\tau \rangle = \langle  2 M^+_{\tau, 0} \phi_\tau, \psi_\tau \rangle = \langle   ((M ')^+_{\tau, 0} +I) \phi_\tau, \psi_\tau \rangle =
\langle   (M ')^+_{\tau, 0} \phi_\tau, \psi_\tau \rangle + \langle \phi_\tau, \psi_\tau \rangle  =
\langle   (M ')^+_{\tau, 0}  M^-_{\tau, 0} \phi_\tau, \psi_\tau \rangle + \langle   (M ')^+_{\tau, 0}  (I-M^-_{\tau, 0}) \phi_\tau, \phi_\tau \rangle + \langle \phi_\tau, \psi_\tau \rangle = 
\langle   M^-_{\tau, 0} \phi_\tau , \psi_\tau \rangle + \langle   (M ')^+_{\tau, 0}  (I-M^-_{\tau, 0}) \phi_\tau, \psi_\tau \rangle + \langle \phi_\tau, \psi_\tau \rangle = 
\langle   d_{\tau}^* \phi_\tau , d_{\tau}^* \psi_\tau \rangle + \langle   (M ')^+_{\tau, 0}  (I-M^-_{\tau, 0}) \phi_\tau, \psi_\tau \rangle + \langle \phi_\tau, \psi_\tau \rangle.
\end{dmath*}
Therefore, for every $\tau \in X (k-1)$,
\begin{dmath*}
\langle d_\tau \phi_\tau d_\tau \psi_\tau \rangle - \dfrac{k}{k+1} \langle \phi_\tau, \psi_\tau \rangle = \langle   d_{\tau}^* \phi_\tau , d_{\tau}^* \psi_\tau \rangle + \langle   (M ')^+_{\tau, 0}  (I-M^-_{\tau, 0}) \phi_\tau, \psi_\tau \rangle + \dfrac{1}{k+1} \langle \phi_\tau, \psi_\tau \rangle
\end{dmath*}
Summing over all $\tau \in X (k-1)$ and applying Proposition \ref{localization proposition} yields
$$\langle d \phi, d \psi \rangle = \langle d^* \phi, d^* \psi \rangle + \langle  \phi, \psi \rangle + \sum_{\tau \in X (k-1)}    \langle (M ')^+_{\tau, 0} (I-M^-_{\tau, 0}) \phi_\tau , \psi_\tau \rangle,$$
as needed.
\end{proof}

In order to deduce mixing from the localization results above, we analyze some special cases of Proposition \ref{M^+ - M^- proposition}.

\begin{lemma}
\label{bound inner-product - non partite case}
Let $0 \leq k  \leq n-1$ and let $\phi, \psi \in C^k (X,\mathbb{R})$. Let $\lambda$ be a constant such that for every $\tau \in X(k-1)$, $\Spec ((M')_{\tau,0}^+) \subseteq [-\lambda, \lambda] \cup \lbrace 1 \rbrace$.
If $\phi$ and $\psi$ are orthogonal, then 
$$\vert \langle (d^* d - d d^*) \phi, \psi \rangle \vert  \leq (k+1) \lambda \Vert \phi \Vert \Vert \psi \Vert.$$
\end{lemma} 

\begin{proof}
Let $\phi, \psi \in C^k (X,\mathbb{R})$ such that $\langle \phi, \psi \rangle =0$. Let $\tau \in X(k-1)$, by the assumption on $\Spec ((M')_{\tau,0}^+)$, 
$$\Vert (M ')^+_{\tau, 0} (I-M^-_{\tau, 0}) \phi_\tau \Vert \leq \lambda \Vert (I-M^-_{\tau, 0}) \phi_\tau \Vert \leq \lambda \Vert \phi_\tau \Vert.$$
Then by Proposition \ref{M^+ - M^- proposition}, 
\begin{dmath*}
\vert \langle (d^* d - d d^*) \phi, \psi \rangle \vert = 
\vert \sum_{\tau \in X (k-1)}    \langle (M ')^+_{\tau, 0} (I-M^-_{\tau, 0}) \phi_\tau , \psi_\tau \rangle \vert \leq \\
 \sum_{\tau \in X (k-1)} \vert \langle (M ')^+_{\tau, 0} (I-M^-_{\tau, 0}) \phi_\tau , \psi_\tau \rangle \vert \leq^{CS}
 \sum_{\tau \in X (k-1)} \Vert (M ')^+_{\tau, 0} (I-M^-_{\tau, 0}) \phi_\tau \Vert \Vert \psi_\tau \Vert \leq \\
 \sum_{\tau \in X (k-1)} \lambda \Vert \phi_\tau \Vert \Vert \psi_\tau \Vert \leq^{CS}  
\lambda \left( \sum_{\tau \in X (k-1)} \Vert \phi_\tau \Vert^2 \right)^{\frac{1}{2}} \left( \sum_{\tau \in X (k-1)} \Vert \psi_\tau \Vert^2 \right)^{\frac{1}{2}} = 
(k+1) \lambda \Vert \phi \Vert \Vert \psi \Vert,
\end{dmath*}
where the last equality is due to Proposition \ref{localization proposition}.
\end{proof}

A version of the above Lemma in the partite case requires the following notation. Let $X$ be $(n+1)$-partite complex with sides $S_0,...,S_n$. Denote $[n] = \lbrace 0,...,n\rbrace$. For $A \subseteq [n]$ such that $\vert A \vert = k+1$, denote 
$$X(S_{i} ; i \in A) = \lbrace \sigma \in X(k) : \forall i \in A, \vert \sigma \cap S_{i} \vert =1 \rbrace.$$

\begin{lemma}
\label{bound inner-product - partite case}
Let $X$ be an $(n+1)$-partite complex with sides $S_0,...,S_n$. Let $0 \leq k  \leq n-1$ and let $\phi, \psi \in C^k (X,\mathbb{R})$. Let $\lambda$ be a constant such that for every $\tau \in X(k-1)$, $\Spec ((M')_{\tau,0}^+) \subseteq [-1, \lambda] \cup \lbrace 1 \rbrace$. If there are $A,B \subseteq [n]$ such that $\vert A \vert = \vert B \vert = k+1, \vert A \cap B \vert < k+1$ and $\supp (\phi) \subseteq X(S_{i} ; i \in A), \supp (\psi) \subseteq X(S_{i} ; i \in B)$, then
$$\vert \langle (d^* d - \dfrac{n+1-k}{n-k} d d^*) \phi, \psi \rangle \vert  \leq (n-k)(k+1) \lambda \Vert \phi \Vert \Vert \psi \Vert.$$
\end{lemma}

\begin{remark}
Note that the bound depends only on the one-sided spectral gap. This is essential, because as we discussed above, in the partite case we do not have a tight lower bound on the spectra in the links. For instance, in $1$-dimensional links of a partite complex $X$, $-1$ is always in the spectrum of the random walk operator.
\end{remark}

\begin{proof}
Let $\phi, \psi$ as above, and let $\tau \in X(k-1)$. Note that $X_\tau$ is a $(n+1-k)$-partite complex with sides $X_\tau (0) \cap S_0,..., X_\tau (0) \cap S_n$ whenever those are not empty sets (note that $k$ of these sets are empty). Therefore $\phi_\tau$ and $\psi_\tau$ are orthogonal, because their supports are disjoint. Let $M^{-,p}_{\tau,0}$ to be the orthogonal projection on the space of function that are constant on the sides of $X_\tau$. We note that since $M^{-,p}_{\tau,0}$ is a projection on a space that contains the constant functions and $M^{-}_{\tau,0}$ is the orthogonal projection on the space of constant functions, it follows that $M^{-}_{\tau,0} M^{-,p}_{\tau,0} = M^{-,p}_{\tau,0} M^{-}_{\tau,0} = M^{-}_{\tau,0}$ and therefore
$$(I-M^{-}_{\tau,0}) = (I-M^{-}_{\tau,0})(M^{-,p}_{\tau,0} + I-M^{-,p}_{\tau,0}) = (M^{-,p}_{\tau,0}-M^{-}_{\tau,0}) + (I-M^{-,p}_{\tau,0}).$$

We will show that under the above assumptions on $\phi, \psi$, 
\begin{equation}
\label{partite lemma eq}
\langle (M')_{\tau,0}^+ (M^{-,p}_{\tau,0} - M^{-}_{\tau,0}) \phi_\tau, \psi_\tau \rangle =  \dfrac{1}{n-k} \langle d^*_\tau \phi_\tau, d^*_\tau \psi_\tau \rangle
\end{equation}
and therefore $\langle (M')_{\tau,0}^+ (I-M^{-}_{\tau,0}) \phi_\tau, \psi_\tau \rangle =  \dfrac{1}{n-k} \langle d^*_\tau \phi_\tau, d^*_\tau \psi_\tau \rangle + \langle (M')_{\tau,0}^+ (I-M^{-,p}_{\tau,0}) \phi_\tau, \psi_\tau \rangle$. We distinguish between two cases: first, if $\tau \neq A \cap B$, then $\phi_\tau \equiv 0$ or $\psi_\tau \equiv 0$ and the equality \ref{partite lemma eq} is obvious (since both sides of the equation are $0$). Assume next that $\tau = A \cap B$. Then there are $i_0, i_1$ such that $\supp (\phi_\tau) \subseteq X_\tau (0) \cap S_{i_0}$ and $\supp (\psi_\tau) \subseteq X_\tau (0) \cap S_{i_1}$. Without loss of generality, we will assume that $\supp (\phi_\tau) \subseteq X_\tau (0) \cap S_{0}$ and $\supp (\psi_\tau) \subseteq X_\tau (0) \cap S_{1}$.  By Lemma \ref{Lemma M' on M^-,p}, 
$$(M')_{\tau,0}^+ M^{-,p}_{\tau,0} \phi_\tau (\lbrace v \rbrace) = \begin{cases}
0 & v \in X_\tau (0) \cap S_0 \\
\frac{n+1-k}{n-k} M_{0,\tau}^{-} \phi_\tau & v \notin X_\tau (0) \cap S_0
\end{cases}.$$ 
Therefore, since $(M')_{\tau,0}^+ M^{-}_{\tau,0} = M^{-}_{\tau,0}$, this yields that
$$(M')_{\tau,0}^+ (M^{-,p}_{\tau,0} - M^{-}_{\tau,0}) \phi_\tau (\lbrace v \rbrace) = \begin{cases}
- M^{-}_{\tau,0} \phi_\tau & v \in X_\tau (0) \cap S_0 \\
\frac{1}{n-k} M_{0,\tau}^{-} \phi_\tau & v \notin X_\tau (0) \cap S_0
\end{cases}.$$ 
Recall that $\psi_\tau$ is supported on $X_\tau (0) \cap S_1$ and therefore
$$\langle (M')_{\tau,0}^+ (M^{-,p}_{\tau,0} - M^{-}_{\tau,0}) \phi_\tau, \psi_\tau \rangle =  \dfrac{1}{n-k} \langle M_{0,\tau}^{-} \phi_\tau,  \psi_\tau \rangle = - \dfrac{n-k-1}{n-k} \langle d^*_\tau \phi_\tau, d^*_\tau \psi_\tau \rangle,$$
as needed.

As noted above, this yields that
$$\langle (M')_{\tau,0}^+ (I-M^{-}_{\tau,0}) \phi_\tau, \psi_\tau \rangle =  \dfrac{1}{n-k} \langle d^*_\tau \phi_\tau, d^*_\tau \psi_\tau \rangle + \langle (M')_{\tau,0}^+ (I-M^{-,p}_{\tau,0}) \phi_\tau, \psi_\tau \rangle.$$
Summing over all $\tau \in X(k-1)$, yields that
\begin{dmath*}
\sum_{\tau \in X (k-1)}    \langle (M ')^+_{\tau, 0} (I-M^-_{\tau, 0}) \phi_\tau , \psi_\tau \rangle =
\sum_{\tau \in X (k-1)} \dfrac{1}{n-k} \langle d^*_\tau \phi_\tau, d^*_\tau \psi_\tau \rangle + \langle (M')_{\tau,0}^+ (I-M^{-,p}_{\tau,0}) \phi_\tau, \psi_\tau \rangle = 
 \dfrac{1}{n-k} \langle d^* \phi, d^* \psi \rangle + \sum_{\tau \in X (k-1)} \langle (M')_{\tau,0}^+ (I-M^{-,p}_{\tau,0}) \phi_\tau, \psi_\tau \rangle = 
 \dfrac{1}{n-k}  \langle d d^* \phi, \psi \rangle + \sum_{\tau \in X (k-1)} \langle (M')_{\tau,0}^+ (I-M^{-,p}_{\tau,0}) \phi_\tau, \psi_\tau \rangle .
\end{dmath*}
Therefore, by Proposition \ref{M^+ - M^- proposition},
$$\langle (d^* d - d d^*) \phi, \psi \rangle  = \dfrac{1}{n-k}  \langle d d^* \phi, \psi \rangle +\sum_{\tau \in X (k-1)} \langle (M')_{\tau,0}^+ (I-M^{-,p}_{\tau,0}) \phi_\tau, \psi_\tau \rangle ,$$
which yields
$$\langle (d^* d - \dfrac{n+1-k}{n-k}d d^*) \phi, \psi \rangle  = \sum_{\tau \in X (k-1)} \langle (M')_{\tau,0}^+ (I-M^{-,p}_{\tau,0}) \phi_\tau, \psi_\tau \rangle.$$

By Proposition \ref{symmetry of spectrum in partite case - prop}, $\Spec ((M')_{\tau,0}^+ (I-M^{-,p}_{\tau,0})) \subseteq [-(n-k) \lambda, \lambda]$, and therefore
\begin{dmath*}
\vert \langle (d^* d - \frac{n+1-k}{n-k} d d^*) \phi, \psi \rangle  \vert = \\
\vert \sum_{\tau \in X (k-1)} \langle (M')_{\tau,0}^+ (I-M^{-,p}_{\tau,0}) \phi_\tau, \psi_\tau \rangle \vert \leq  \\
 \sum_{\tau \in X (k-1)} \vert \langle (M')_{\tau,0}^+ (I-M^{-,p}_{\tau,0}) \phi_\tau, \psi_\tau \rangle \vert \leq^{CS} \\
\sum_{\tau \in X (k-1)} \Vert (M')_{\tau,0}^+ (I-M^{-,p}_{\tau,0}) \phi_\tau \Vert \Vert \psi_\tau  \Vert \leq \\
\sum_{\tau \in X (k-1)} (n-k) \lambda \Vert \phi_\tau \Vert \Vert \psi_\tau \Vert \leq^{CS} \\
 (n-k) \lambda \left(\sum_{\tau \in X (k-1)} \Vert \phi_\tau \Vert^2 \right)^{\frac{1}{2}} \left(\sum_{\tau \in X (k-1)} \Vert \psi_\tau \Vert^2 \right)^{\frac{1}{2}} =
 (n-k) (k+1) \lambda \Vert \phi \Vert \Vert \psi \Vert.
\end{dmath*}
\end{proof}

\section{Random walks on sets of simplices determined by vertex sets}
\label{Random walks on sets of simplices determined by vertex sets sec}

Below, we will explore the connections between $d^*d$ and $d d^*$ when those are restricted to sets of simplices determined by vertex sets. 

We will use the following convention: given operators $A_0,...,A_k$, $\prod_{i=0}^k A_i = A_k A_{k-1} ... A_0$, .i.e., the product notation refers to multiplying from right to left.

Let $X$ be pure $n$-dimensional weighted simplicial complex and let $U_0,...,U_k \subseteq X(0)$ be disjoint sets. Denote $P_{X(U_{0},...,U_{k})} : C^k (X,\mathbb{R}) \rightarrow C^k (X,\mathbb{R})$ to be the projection on $k$-cochains supported on $X(U_{0},...,U_{k})$, i.e., for every $\sigma \in X(k)$,
$$P_{X(U_{0},...,U_{k})} \phi (\sigma) = \begin{cases}
\phi (\sigma) & \sigma \in X(U_{0},...,U_{k}) \\
0 & \sigma \notin  X(U_{0},...,U_{k})
\end{cases}.$$

\begin{lemma}
\label{P d^* d P lemma}
Let $X$ as above and let $0 \leq k \leq n-1$. For $U_0,...,U_{k+1} \subseteq X(0)$ pairwise disjoint sets, the following equality holds:
$$P_{X(U_1,...,U_{k+1})} d^*_k P_{X(U_0,...,U_{k+1})} d_k P_{X(U_0,...,U_{k})} = P_{X(U_1,...,U_{k+1})} d^*_k d_k P_{X(U_0,...,U_{k})}.$$
\end{lemma}

\begin{proof}
We will show that for every $\phi \in C^k (X, \mathbb{R})$ and every $\sigma \in X(k)$,
$$(P_{X(U_1,...,U_{k+1})} d^*_k P_{X(U_0,...,U_{k+1})} d_k P_{X(U_0,...,U_{k})} \phi ) (\sigma) = (P_{X(U_1,...,U_{k+1})} d^*_k d_k P_{X(U_0,...,U_{k})} \phi ) (\sigma).$$
Let $\phi \in C^k (X, \mathbb{R})$ and let $\sigma \in X(k)$. If $\sigma \notin X(U_1,...,U_{k+1})$, then the equality holds trivially, because both sides are $0$. Assume that $\sigma = \lbrace u_1,...,u_{k+1} \rbrace$ and $u_i \in U_i$ for every $1 \leq i \leq k+1$. Note that Lemma \ref{d* lemma}, for every $\psi \in C^{k+1} (X, \mathbb{R})$, 
$$(d^*_k \psi) (\lbrace u_1,...,u_{k+1} \rbrace) = \sum_{v \in X(0), \lbrace v, u_1,...,u_{k+1} \rbrace \in X(k+1)} \dfrac{m(\lbrace v, u_1,...,u_{k+1} \rbrace)}{m(\lbrace u_1,...,u_{k+1} \rbrace)} \psi (\lbrace v, u_1,...,u_{k+1} \rbrace).$$
For $\psi = d_k P_{X(U_0,...,U_{k})} \phi$, we note that for every $\lbrace v, u_1,...,u_{k+1} \rbrace \in X(k+1)$, 
\begin{dmath*}
(d_k P_{X(U_0,...,U_{k})} \phi)  (\lbrace v, u_1,...,u_{k+1} \rbrace) = \sum_{\tau \subseteq \lbrace v, u_1,...,u_{k+1} \rbrace, \tau \in X(k)} P_{X(U_1,...,U_{0})} \phi (\tau) = \begin{cases}
0 & v \notin U_0 \\
\phi (\lbrace v, u_1,...,u_{k} \rbrace) & v \in U_0
\end{cases}.
\end{dmath*}
Therefore, for every $\lbrace v, u_1,...,u_{k+1} \rbrace \in X(k+1)$, 
$$d_k P_{X(U_0,...,U_{k})} \phi (\lbrace v, u_1,...,u_{k+1} \rbrace) = P_{X(U_0,...,U_{k+1})} d_k P_{X(U_0,...,U_{k})} \phi (\lbrace v, u_1,...,u_{k+1} \rbrace),$$
and therefore
$$(d^*_k d_k P_{X(U_0,...,U_{k})} \phi) (\lbrace u_1,...,u_{k+1} \rbrace) = (d^*_k P_{X(U_0,...,U_{k+1})} d_k P_{X(U_0,...,U_{k})} \phi) (\lbrace u_1,...,u_{k+1} \rbrace),$$
as needed.
\end{proof}

\begin{lemma}
\label{P d d^* P lemma}
Let $X$ as above and let $0 \leq k \leq n-2$. For $U_0,...,U_{k+2} \subseteq X(0)$ pairwise disjoint sets, the following equality holds:
$$P_{X(U_1,...,U_{k+2})} d_k P_{X(U_1,...,U_{k+1})} d_k^* P_{X(U_0,...,U_{k+1})} = P_{X(U_1,...,U_{k+2})} d_k d_k^* P_{X(U_0,...,U_{k+1})}.$$
\end{lemma}

\begin{proof}
The proof is very similar to the proof of Lemma \ref{P d^* d P lemma} above. As above it is enough to show that for every $\psi \in C^{k+1} (X,\mathbb{R})$ and every $\lbrace u_1,...,u_{k+2} \rbrace \in X(k+1)$ such that $u_i \in U_i$ it holds that 
$$(d_k P_{X(U_1,...,U_{k+1})} d_k^* P_{X(U_0,...,U_{k+1})} \psi) (\lbrace u_1,...,u_{k+2} \rbrace)  = (d_k d_k^* P_{X(U_0,...,U_{k+1})} \psi )(\lbrace u_1,...,u_{k+2} \rbrace).$$
Note that
$$\supp (d_k^* P_{X(U_0,...,U_{k+1})} \psi) \subseteq \bigcup_{i=0}^{k+1} X(U_0,...\widehat{U_i},...,U_{k+1}).$$ 
Therefore 
\begin{dmath*}
(d_k d_k^* P_{X(U_0,...,U_{k+1})} \psi ) (\lbrace u_1,...,u_{k+2} \rbrace) =  
\sum_{i=1}^{k+2} (d_k^* P_{X(U_0,...,U_{k+1})} \psi ) (\lbrace u_1,...,\widehat{u_i},...,u_{k+2} \rbrace) =
d_k^* P_{X(U_0,...,U_{k+1})} \psi (\lbrace u_1,...,u_{k+1} \rbrace) =
P_{X(U_1,...,U_{k+1})} d_k^* P_{X(U_0,...,U_{k+1})} \psi (\lbrace u_1,...,u_{k+1} \rbrace) =
(d_k P_{X(U_1,...,U_{k+1})} d_k^* P_{X(U_0,...,U_{k+1})} \psi ) (\lbrace u_1,...,u_{k+2} \rbrace),
\end{dmath*}
as needed.
\end{proof}

\begin{lemma}
\label{product lemma}
Let $X$ as above and let $0 \leq k \leq n$ be a constant. For any disjoint sets $U_0,...,U_n \subseteq X(0)$ the following holds:
\begin{align*}
 P_{X(U_{n-k},...,U_{n})} d_k^* \left(\prod_{i=1}^{n-k-1} P_{X(U_{i},...,U_{k+1+i})} d_k d^*_k \right) P_{X(U_0,...,U_{k+1})} d_k P_{X(U_0,...,U_{k})} = \\
\left(\prod_{i=1}^{n-k} P_{X(U_{i},...,U_{k+i})} d_k^* d_k \right) P_{X(U_{0},...,U_{k})}.
\end{align*}
\end{lemma}

\begin{proof}
Let $U_0,...,U_n \subseteq X(0)$ as above, then
\begin{dmath*}
P_{X(U_{n-k},...,U_{n})} d_k^* \left(\prod_{i=1}^{n-k-1} P_{X(U_{i},...,U_{k+1+i})} d_k d^*_k \right) P_{X(U_0,...,U_{k+1})} d_k P_{X(U_0,...,U_{k})} = 
P_{X(U_{n-k},...,U_{n})} d_k^* \left(\prod_{i=1}^{n-k-1} P_{X(U_{i},...,U_{k+1+i})} d_k d^*_k P_{X(U_{i-1},...,U_{k+i})} \right)  d_k P_{X(U_0,...,U_{k})} =^{\text{Lemma } \ref{P d d^* P lemma}} 
P_{X(U_{n-k},...,U_{n})} d_k^* \left(\prod_{i=1}^{n-k-1} P_{X(U_{i},...,U_{k+1+i})} d_k P_{X(U_{i},...,U_{k+i})} d^*_k P_{X(U_{i-1},...,U_{k+i})} \right)  d_k P_{X(U_0,...,U_{k})} = 
\prod_{i=1}^{n-k} P_{X(U_{i},...,U_{k+i})} d_k^* P_{X(U_{i-1},...,U_{k+i})} d_k P_{X(U_{i-1},...,U_{k-1+i})} =^{\text{Lemma } \ref{P d d^* P lemma}} 
\prod_{i=1}^{n-k} P_{X(U_{i},...,U_{k+i})} d_k^* d_k P_{X(U_{i-1},...,U_{k-1+i})} =
\left(\prod_{i=1}^{n-k} P_{X(U_{i},...,U_{k+i})} d_k^* d_k \right) P_{X(U_{0},...,U_{k})}.
\end{dmath*}
\end{proof}

\begin{corollary}
\label{inner-product coro}
Let $X$ as above and let $0 \leq k \leq n-2$ be a constant. For any disjoint sets $U_0,...,U_n \subseteq X(0)$ the following holds:
\begin{align*}
\left\langle \left(\prod_{i=1}^{n-k-1} P_{X(U_{i},...,U_{k+1+i})} d_k d^*_k \right) \chi_{X(U_0,...,U_{k+1})}, \chi_{X(U_{n-k-1},...,U_{n})} \right\rangle = \\
 \left\langle \left(\prod_{i=1}^{n-k} P_{X(U_{i},...,U_{k+i})} d_k^* d_k \right) \chi_{X(U_0,...,U_{k})}, \chi_{X(U_{n-k},...,U_{n})} \right\rangle.
\end{align*}

\end{corollary}

\begin{proof}
Note that 
$$P_{X(U_0,...,U_{k+1})} d_k P_{X(U_0,...,U_{k})} \chi_{X(U_0,...,U_{k})} =  \chi_{X(U_0,...,U_{k+1})},$$
$$P_{X(U_{n-k-1},...,U_{n})} d_k P_{X(U_{n-k},...,U_{n})} \chi_{X(U_{n-k},...,U_{n})} =  \chi_{X(U_{n-k-1},...,U_{n})}.$$
Therefore
\begin{dmath*}
\left\langle \left(\prod_{i=1}^{n-k-1} P_{X(U_{i},...,U_{k+1+i})} d_k d^*_k \right) \chi_{X(U_0,...,U_{k+1})}, \chi_{X(U_{n-k-1},...,U_{n})} \right\rangle = 
\left\langle \left(\prod_{i=1}^{n-k-1} P_{X(U_{i},...,U_{k+1+i})} d_k d^*_k \right) P_{X(U_0,...,U_{k+1})} d_k P_{X(U_0,...,U_{k})} \chi_{X(U_0,...,U_{k})}, \\
P_{X(U_{n-k-1},...,U_{n})} d_k P_{X(U_{n-k},...,U_{n})} \chi_{X(U_{n-k},...,U_{n})} \right\rangle = 
\left\langle P_{X(U_{n-k},...,U_{n})} d_k^* P_{X(U_{n-k-1},...,U_{n})} \left(\prod_{i=1}^{n-k-1} P_{X(U_{i},...,U_{k+1+i})} d_k d^*_k \right) \\
P_{X(U_0,...,U_{k+1})} d_k P_{X(U_0,...,U_{k})} \chi_{X(U_0,...,U_{k})}, 
 \chi_{X(U_{n-k},...,U_{n})} \right\rangle, 
\end{dmath*}
and the corollary follows from Lemma \ref{product lemma}. 
\end{proof}

While the above corollary described a connection between $d_k^* d_k$ and $d_k d^*_k$, the results below describe a connection between $d_k^* d_k$ and $d_{k-1} d^*_{k-1}$ under assumptions on the local spectra at the links.

\begin{lemma}
\label{product norm bound - non partite case}
Let $X$ be as above and let $0 \leq k \leq n-1$. Assume that there is a constant $0< \lambda < 1$ such that for every $\tau \in X(k-1)$, $\Spec ((M')_{\tau,0}^+) \subseteq [-\lambda, \lambda] \cup \lbrace 1 \rbrace$, then for every pairwise disjoint sets $U_0,...,U_n \subseteq X(0)$,
\begin{dmath*}
\left\Vert \prod_{i=1}^{n-k} P_{X(U_{i},...,U_{k+i})} d_k^* d_k P_{X(U_{i-1},...,U_{k-1+i})}  - {\prod_{i=1}^{n-k} P_{X(U_{i},...,U_{k+i})} d_{k-1} d_{k-1}^* P_{X(U_{i-1},...,U_{k-1+i})}} \right\Vert \leq \lambda ((k+1)(k+2)^{n-k} - (k+1)^{n-k+1} ).
\end{dmath*}
\end{lemma}

\begin{proof}
We first note that for every $1 \leq i \leq n-k$ and for every $\phi, \psi \in C^k (X,\mathbb{R})$, $P_{X(U_{i-1},...,U_{k-1+i})} \phi$ is orthogonal to $P_{X(U_{i},...,U_{k+i})} \psi$ and therefore by Lemma \ref{bound inner-product - non partite case}, we have that
\begin{dmath*}
{\left\vert \left\langle P_{X(U_{i},...,U_{k+i})} \left( d_k^* d_k -  d_{k-1} d_{k-1}^* \right) P_{X(U_{i-1},...,U_{k-1+i})} \phi, \psi \right\rangle \right\vert } = 
\left\vert \left\langle \left( d_k^* d_k -  d_{k-1} d_{k-1}^* \right) P_{X(U_{i-1},...,U_{k-1+i})} \phi, P_{X(U_{i},...,U_{k+i})} \psi \right\rangle \right\vert \leq 
(k+1) \lambda \Vert P_{X(U_{i-1},...,U_{k-1+i})} \phi \Vert \Vert P_{X(U_{i},...,U_{k+i})} \psi \Vert \leq 
(k+1) \lambda \Vert \phi \Vert \Vert \psi \Vert,
\end{dmath*}
which implies that
$$\left\Vert P_{X(U_{i-1},...,U_{k-1+i})} \left( d_k^* d_k -  d_{k-1} d_{k-1}^* \right) P_{X(U_{i},...,U_{k+i})} \right\Vert \leq (k+1) \lambda.$$
By the triangle inequality,
\begin{dmath*}
\left\Vert \prod_{i=1}^{n-k} P_{X(U_{i},...,U_{k+i})} d_k^* d_k P_{X(U_{i-1},...,U_{k-1+i})}  - {\prod_{i=1}^{n-k} P_{X(U_{i},...,U_{k+i})} d_{k-1} d_{k-1}^* P_{X(U_{i-1},...,U_{k-1+i})}} \right\Vert \leq \\
\sum_{j=1}^{n-k} \left\Vert \left( \prod_{i=j+1}^{n-k} P_{X(U_{i},...,U_{k+i})}  d_{k-1} d_{k-1}^* P_{X(U_{i-1},...,U_{k-1+i})} \right) \\
\left( P_{X(U_{j},...,U_{k+j})} \left( d_k^* d_k -  d_{k-1} d_{k-1}^* \right) P_{X(U_{j-1},...,U_{k-1+j})} \right) \\
 \left( \prod_{i=1}^{j-1} P_{X(U_{i},...,U_{k+i})}  d_{k}^* d_{k} P_{X(U_{i-1},...,U_{k-1+i})} \right) \right\Vert \leq^{\text{Corollary } \ref{norm of d^* d, d d^* coro}} 
 (k+1) \lambda \sum_{j=1}^{n-k} (k+2)^{j-1} (k+1)^{n-k-j} = \lambda ((k+1)(k+2)^{n-k} - (k+1)^{n-k+1} ).
\end{dmath*}
\end{proof}

\begin{corollary}
\label{product norm bound of indicator functions - non partite case coro}
Let $X$ be as above and let $0 \leq k \leq n-1$. Assume that there is a constant $0< \lambda < 1$ such that for every $\tau \in X(k-1)$, $\Spec ((M')_{\tau,0}^+) \subseteq [-\lambda, \lambda] \cup \lbrace 1 \rbrace$, then for every pairwise disjoint sets $U_0,...,U_n \subseteq X(0)$,
\begin{dmath*}
\left\vert \left\langle  \left(\left( \prod_{i=1}^{n-k} P_{X(U_{i},...,U_{k+i})} d_k^* d_k \right) - \left(\prod_{i=1}^{n-k} P_{X(U_{i},...,U_{k+i})} d_{k-1} d_{k-1}^* \right) \right) \chi_{X(U_0,...,U_k)}, \chi_{X(U_{n-k},...,U_n)} \right\rangle \right\vert \leq 
\lambda ((k+1)(k+2)^{n-k} - (k+1)^{n-k+1} ) \sqrt{m(X(U_0,...,U_k)) m(X(U_{n-k},...,U_n))} \leq 
\lambda ((k+1)(k+2)^{n-k} - (k+1)^{n-k+1} ) \sqrt{m(U_0) m(U_n)}.
\end{dmath*}
\end{corollary}

\begin{proof}
Note that $\chi_{X(U_0,...,U_k)} = P_{X(U_0,...,U_k)} \chi_{X(U_0,...,U_k)}$ and therefore
\begin{dmath*}
\left\vert \left\langle  \left(\left( \prod_{i=1}^{n-k} P_{X(U_{i},...,U_{k+i})} d_k^* d_k \right) - {\left(\prod_{i=1}^{n-k} P_{X(U_{i},...,U_{k+i})} d_{k-1} d_{k-1}^* \right)} \right) \chi_{X(U_0,...,U_k)}, \chi_{X(U_{n-k},...,U_n)} \right\rangle \right\vert = \\
\left\vert \left\langle  \left(\left( \prod_{i=1}^{n-k} P_{X(U_{i},...,U_{k+i})} d_k^* d_k P_{X(U_{i-1},...,U_{k-1+i})}\right) - {\left(\prod_{i=1}^{n-k} P_{X(U_{i},...,U_{k+i})} d_{k-1} d_{k-1}^* P_{X(U_{i-1},...,U_{k-1+i})} \right)} \right) \chi_{X(U_0,...,U_k)}, \chi_{X(U_{n-k},...,U_n)} \right\rangle \right\vert \leq^{CS}
\left\Vert \left( \prod_{i=1}^{n-k} P_{X(U_{i},...,U_{k+i})} d_k^* d_k P_{X(U_{i-1},...,U_{k-1+i})}\right) - \\
\left(\prod_{i=1}^{n-k} P_{X(U_{i},...,U_{k+i})} d_{k-1} d_{k-1}^* P_{X(U_{i-1},...,U_{k-1+i})} \right) \right\Vert 
 \Vert \chi_{X(U_0,...,U_k)} \Vert \Vert \chi_{X(U_{n-k},...,U_n)} \Vert \leq^{\text{Lemma } \ref{product norm bound - non partite case}} \\
\lambda ((k+1)(k+2)^{n-k} - (k+1)^{n-k+1} ) \sqrt{m(X(U_0,...,U_k)) m(X(U_{n-k},...,U_n))}.
\end{dmath*}
The fact that $m(X(U_0,...,U_k)) \leq m(U_0), m(X(U_{n-k},...,U_n)) \leq m(U_n)$ follows from the definition of $m$.
\end{proof}

In the partite case, the analogue to Lemma \ref{product norm bound - non partite case} reads as follows:
\begin{lemma}
\label{product norm bound - partite case}
Let $X$ be as above and let $0 \leq k \leq n-1$. Assume that $X$ is $(n+1)$-partite with sides $S_0,...,S_n$. Also assume that there is a constant $0< \lambda < 1$ such that for every $\tau \in X(k-1)$, $\Spec ((M')_{\tau,0}^+) \subseteq [-1, \lambda] \cup \lbrace 1 \rbrace$, then for every sets $U_0 \subseteq S_0,...,U_n \subseteq S_n$,
\begin{dmath*}
\left\Vert \prod_{i=1}^{n-k} P_{X(U_{i},...,U_{k+i})} d_k^* d_k P_{X(U_{i-1},...,U_{k-1+i})}  - \prod_{i=1}^{n-k} P_{X(U_{i},...,U_{k+i})} \left( \dfrac{n+1-k}{n-k} d_{k-1} d_{k-1}^* \right) P_{X(U_{i-1},...,U_{k-1+i})} \right\Vert \leq \\
 \lambda (n-k)((k+1)(k+2)^{n-k} - (k+1)^{n-k+1} ).
\end{dmath*}
\end{lemma}

\begin{proof}
The proof is very similar to that of Lemma \ref{product norm bound - non partite case}, when Lemma \ref{bound inner-product - partite case} is used instead of Lemma \ref{bound inner-product - non partite case} - we leave the details to the reader.
\end{proof}

\begin{remark}
One can actually derive a slightly better bound above: one can show that
\begin{dmath*}
\left\Vert \prod_{i=1}^{n-k} P_{X(U_{i},...,U_{k+i})} d_k^* d_k P_{X(U_{i-1},...,U_{k-1+i})}  - \prod_{i=1}^{n-k} P_{X(U_{i},...,U_{k+i})} \left( \dfrac{n+1-k}{n-k} d_{k-1} d_{k-1}^* \right) P_{X(U_{i-1},...,U_{k-1+i})} \right\Vert \leq \\
 \lambda \dfrac{(k+1)(k+2)^{n-k} (n-k)^{n-k} = (k+1)^{n+1-k} (n+1-k)^{n-k}}{(n-k)^{n-k-2} (n-2k-1)}.
\end{dmath*}
However, for the sake of simplicity, we will use the simpler bound stated in Lemma \ref{product norm bound - partite case}.
\end{remark}

Similarly to Corollary \ref{product norm bound of indicator functions - non partite case coro}, the above Lemma in the partite case implies the following:
\begin{corollary}
\label{product norm bound of indicator functions - partite case coro}
Let $X$ be as above and let $0 \leq k \leq n-1$. Assume that $X$ is $(n+1)$-partite with sides $S_0,...,S_n$. Also assume that there is a constant $0< \lambda < 1$ such that for every $\tau \in X(k-1)$, $\Spec ((M')_{\tau,0}^+) \subseteq [-1, \lambda] \cup \lbrace 1 \rbrace$, then for every sets $U_0 \subseteq S_0,...,U_n \subseteq S_n$,
\begin{dmath*}
\left\vert \left\langle  \left(\left( \prod_{i=1}^{n-k} P_{X(U_{i},...,U_{k+i})} d_k^* d_k \right) - \left(\prod_{i=1}^{n-k} P_{X(U_{i},...,U_{k+i})} \left( \dfrac{n+1-k}{n-k} d_{k-1} d_{k-1}^* \right) \right) \right) \chi_{X(U_0,...,U_k)}, \chi_{X(U_{n-k},...,U_n)} \right\rangle \right\vert \leq 
\lambda (n-k) ((k+1)(k+2)^{n-k} - (k+1)^{n-k+1} ) \sqrt{m(X(U_0,...,U_k)) m(X(U_{n-k},...,U_n))} \leq
\lambda (n-k) ((k+1)(k+2)^{n-k} - (k+1)^{n-k+1} ) \sqrt{m(U_0) m(U_n)}.
\end{dmath*}
\end{corollary}

\begin{proof}
The proof is similar to that of Corollary \ref{product norm bound of indicator functions - non partite case coro} and we leave it to the reader.
\end{proof}

Last, below we will need this following additional Lemma:
\begin{lemma}
\label{M^-_0 on U_0,...,U_n lemma}
Let $X$ be as above. For every pairwise disjoint sets $U_0,...,U_n \subseteq X(0)$, the following holds:
$$\left\langle \left( \prod_{i=1}^{n} P_{X(U_{i})} d_{-1} d_{-1}^* \right) \chi_{X(U_0)}, \chi_{X(U_n)} \right\rangle = \dfrac{m(U_0)...m(U_n)}{(m(X(0)))^n}.$$
\end{lemma}

\begin{proof}
We note that by the definition of $d_{-1}$ and by the computation of $d_{-1}^*$ in Lemma \ref{d* lemma},  it holds that for every $i$
$$(d_{-1} d_{-1}^*) \chi_{X(U_k)} \equiv \dfrac{m(U_i)}{m(X(0))}.$$
Therefore, by induction on $k$, for every $1 \leq k \leq n$,
$$\left( \prod_{i=1}^{k} P_{X(U_{i})} d_{-1} d_{-1}^* \right) \chi_{X(U_0)} = \dfrac{m(U_0)... m(U_{k-1})}{(m(X(0)))^{k}} \chi_{X(U_k)},$$
and in particular 
$$\left( \prod_{i=1}^{n} P_{X(U_{i})} d_{-1} d_{-1}^* \right) \chi_{X(U_0)} = \dfrac{m(U_0)... m(U_{n-1})}{(m(X(0)))^{n}} \chi_{X(U_n)}.$$
Therefore 
\begin{dmath*}
\left\langle \left( \prod_{i=1}^{n} P_{X(U_{i})} d_{-1} d_{-1}^* \right) \chi_{X(U_0)}, \chi_{X(U_n)} \right\rangle = 
\dfrac{m(U_0)... m(U_{n-1})}{(m(X(0)))^{n}} \left\langle \chi_{X(U_n)}, \chi_{X(U_n)} \right\rangle =
\dfrac{m(U_0)...m(U_n)}{(m(X(0)))^n},
\end{dmath*}
as needed.
\end{proof}

\section{Mixing}
\label{Mixing sec}
\begin{theorem}[Mixing Theorem - non-partite complexes]
\label{Mixing Theorem - non-partite complexes}
Let $X$ be as above and assume that there is a constant $0< \lambda < 1$ such that for every $0 \leq k \leq n-1$ and every $\tau \in X(k-1)$, $\Spec ((M')_{\tau,0}^+) \subseteq [-\lambda, \lambda] \cup \lbrace 1 \rbrace$, then for every pairwise disjoint sets $U_0,...,U_n \subseteq X(0)$,
$$\left\vert m(X(U_0,...,U_n)) - \dfrac{m(U_0) m (U_1) ... m(U_n)}{m(X(0))^{n}} \right\vert \leq C_n \lambda \min_{0 \leq i < j \leq n} \sqrt{m(U_i) m(U_j)},$$
where 
$$C_n = \sum_{k=0}^{n-1} ((k+1)(k+2)^{n-k} - (k+1)^{n-k+1} ).$$
\end{theorem}

\begin{proof}
Let $U_0,...,U_n \subseteq X(0)$ be pairwise disjoint sets. Without loss of generality, we can assume that $\min_{0 \leq i < j \leq n} \sqrt{m(U_i) m(U_j)} =  \sqrt{m(U_0) m(U_n)}$. We note that for every $\sigma \in X(n)$,
$$(d_{n-1} \chi_{X(U_0,...,U_{n-1})} (\sigma)) (d_{n-1} \chi_{X(U_1,...,U_{n})} (\sigma))= \begin{cases}
1 & \sigma \in X(U_0,...,U_n) \\
0 & \text{otherwise}
\end{cases}.$$
Therefore
\begin{dmath*}
\langle (d^*_{n-1} d_{n-1}) \chi_{X(U_0,...,U_{n-1})}, \chi_{X(U_1,...,U_{n})} \rangle = \langle d_{n-1} \chi_{X(U_0,...,U_{n-1})}, d_{n-1} \chi_{X(U_1,...,U_{n})} \rangle = m(X(U_0,...,U_n)).
\end{dmath*}
By Lemma \ref{M^-_0 on U_0,...,U_n lemma} it follows that
\begin{dmath*}
m(X(U_0,...,U_n)) - \dfrac{m(U_0) m (U_1) ... m(U_n)}{m(X(0))^{n}}  =  {\left\langle (d^*_{n-1} d_{n-1}) \chi_{X(U_0,...,U_{n-1})}, \chi_{X(U_1,...,U_{n})} \right\rangle - \left\langle \left( \prod_{i=1}^{n} P_{X(U_{i})} d_{-1} d_{-1}^* \right) \chi_{X(U_0)}, \chi_{X(U_n)} \right\rangle.}
\end{dmath*}
Next, we will show by induction that for every $0 \leq l \leq n-1$,
\begin{dmath*}
\left\langle \left(\prod_{i=1}^{n-l} P_{X(U_{i},...,U_{l+i})} d_l^* d_l \right) \chi_{X(U_0,...,U_{l})}, \chi_{X(U_{n-l},...,U_{n})} \right\rangle - 
\left\langle \left( \prod_{i=1}^{n} P_{X(U_{i})} d_{-1} d_{-1}^* \right) \chi_{X(U_0)}, \chi_{X(U_n)} \right\rangle =
\sum_{k=0}^l \left\langle  \left(\left( \prod_{i=1}^{n-k} P_{X(U_{i},...,U_{k+i})} d_k^* d_k \right) - \left(\prod_{i=1}^{n-k} P_{X(U_{i},...,U_{k+i})} d_{k-1} d_{k-1}^* \right) \right) \chi_{X(U_0,...,U_k)}, \chi_{X(U_{n-k},...,U_n)} \right\rangle. 
\end{dmath*}
For $l=0$ this equality is trivial. Assume that it holds for $l$, then for $l+1$, we have that
\begin{dmath*}
\left\langle \left(\prod_{i=1}^{n-l} P_{X(U_{i},...,U_{l+1+i})} d_{l+1}^* d_{l+1} \right) \chi_{X(U_0,...,U_{l+1})}, \chi_{X(U_{n-l-1},...,U_{n})} \right\rangle \\
- \left\langle \left( \prod_{i=1}^{n} P_{X(U_{i})} d_{-1} d_{-1}^* \right) \chi_{X(U_0)}, \chi_{X(U_n)} \right\rangle = \\
\left\langle \left(\prod_{i=1}^{n-l} P_{X(U_{i},...,U_{l+1+i})} d_{l+1}^* d_{l+1} \right) \chi_{X(U_0,...,U_{l+1})}, \chi_{X(U_{n-l-1},...,U_{n})} \right\rangle \\
- \left\langle \left(\prod_{i=1}^{n-l} P_{X(U_{i},...,U_{l+1+i})} d_{l} d_{l}^* \right) \chi_{X(U_0,...,U_{l+1})}, \chi_{X(U_{n-l-1},...,U_{n})} \right\rangle \\
+ \left\langle \left(\prod_{i=1}^{n-l} P_{X(U_{i},...,U_{l+1+i})} d_{l} d_{l}^* \right) \chi_{X(U_0,...,U_{l+1})}, \chi_{X(U_{n-l-1},...,U_{n})} \right\rangle \\
{- \left\langle \left( \prod_{i=1}^{n} P_{X(U_{i})} d_{-1} d_{-1}^* \right) \chi_{X(U_0)}, \chi_{X(U_n)} \right\rangle =^{\text{Corollary } \ref{inner-product coro}}} \\
\left\langle \left( \left(\prod_{i=1}^{n-l} P_{X(U_{i},...,U_{l+1+i})} d_{l+1}^* d_{l+1} \right) - \left(\prod_{i=1}^{n-l} P_{X(U_{i},...,U_{l+1+i})} d_{l} d_{l}^* \right) \right) \chi_{X(U_0,...,U_{l+1})}, \chi_{X(U_{n-l-1},...,U_{n})} \right\rangle \\
 + \left\langle \left(\prod_{i=1}^{n-l} P_{X(U_{i},...,U_{l+i})} d_l^* d_l \right) \chi_{X(U_0,...,U_{l})}, \chi_{X(U_{n-l},...,U_{n})} \right\rangle \\
 - 
\left\langle \left( \prod_{i=1}^{n} P_{X(U_{i})} d_{-1} d_{-1}^* \right) \chi_{X(U_0)}, \chi_{X(U_n)} \right\rangle,
\end{dmath*}
and the assertion follows by the induction assumption. Therefore, for $l=n-1$, it follows that
\begin{dmath*}
m(X(U_0,...,U_n)) - \dfrac{m(U_0) m (U_1) ... m(U_n)}{m(X(0))^{n}}  =  
{\left\langle (d^*_{n-1} d_{n-1}) \chi_{X(U_0,...,U_{n-1})}, \chi_{X(U_1,...,U_{n})} \right\rangle - \left\langle \left( \prod_{i=1}^{n} P_{X(U_{i})} d_{-1} d_{-1}^* \right) \chi_{X(U_0)}, \chi_{X(U_n)} \right\rangle} = 
\sum_{k=0}^{n-1} \left\langle  \left(\left( \prod_{i=1}^{n-k} P_{X(U_{i},...,U_{k+i})} d_k^* d_k \right) - \left(\prod_{i=1}^{n-k} P_{X(U_{i},...,U_{k+i})} d_{k-1} d_{k-1}^* \right) \right) \chi_{X(U_0,...,U_k)}, \chi_{X(U_{n-k},...,U_n)} \right\rangle. 
\end{dmath*}
Therefore
\begin{dmath*}
{\left\vert m(X(U_0,...,U_n)) - \dfrac{m(U_0) m (U_1) ... m(U_n)}{m(X(0))^{n}} \right\vert } \leq  
{\sum_{k=0}^{n-1} \left\vert \left\langle  \left(\left( \prod_{i=1}^{n-k} P_{X(U_{i},...,U_{k+i})} d_k^* d_k \right) - \left(\prod_{i=1}^{n-k} P_{X(U_{i},...,U_{k+i})} d_{k-1} d_{k-1}^* \right) \right) \chi_{X(U_0,...,U_k)}, \chi_{X(U_{n-k},...,U_n)} \right\rangle \right\vert} \leq^{\text{Corollary } \ref{product norm bound of indicator functions - non partite case coro}} 
\sum_{k=0}^{n-1} \lambda ((k+1)(k+2)^{n-k} - (k+1)^{n-k+1} ) \sqrt{m(U_0) m(U_n)} = C_n \lambda  \sqrt{m(U_0) m(U_n)}.
\end{dmath*}
\end{proof}

The above Mixing Theorem can be simplified in the case where $X$ is regular and the weight $m$ is the homogeneous weight - this simplified version is Theorem \ref{mixing thm for general complexes - intro} that appeared in the introduction:
\begin{corollary}
\label{mixing for regular coro}
Let $X$ be a pure $n$-dimensional simplicial complex with the homogeneous weight function. Assume that $X$ is $K$-regular in the following sense: for every $\lbrace v \rbrace \in X(0)$, $v$ is contained in exactly $K$ $n$-dimensional simplices of $X$, i.e., $m(\lbrace v \rbrace) = n! K$. Assume that there is a constant $0< \lambda < 1$ such that for every $0 \leq k \leq n-1$ and every $\tau \in X(k-1)$, $\Spec ((M')_{\tau,0}^+) \subseteq [-\lambda, \lambda] \cup \lbrace 1 \rbrace$, then for every pairwise disjoint sets $U_0,...,U_n \subseteq X(0)$,
$$\left\vert \vert X(U_0,...,U_n) \vert - \dfrac{n! K}{\vert X(0) \vert^n} \vert U_0 \vert \vert U_1 \vert ... \vert U_n \vert \right\vert \leq C_n n! \lambda K  \min_{0 \leq i < j \leq n} \sqrt{ \vert U_i \vert \vert U_j \vert},$$
and
$$\left\vert \frac{\vert X(U_0,...,U_n) \vert}{\vert X(n) \vert} - (n+1)! \dfrac{\vert U_0 \vert \vert U_1 \vert ... \vert U_n \vert}{\vert X(0) \vert^{n+1}} \right\vert \leq C_n (n+1)! \lambda  \min_{0 \leq i < j \leq n} \sqrt{ \frac{\vert U_i \vert \vert U_j \vert}{\vert X(0) \vert^2}},$$
where 
$$C_n = \sum_{k=0}^{n-1} ((k+1)(k+2)^{n-k} - (k+1)^{n-k+1} ).$$ 
\end{corollary}

\begin{proof}
The corollary follows from Theorem \ref{Mixing Theorem - non-partite complexes} by the following equalities:
$$m(X(U_0,...,U_n)) = \vert X(U_0,...,U_n) \vert, m(X(0)) = (n+1)! \vert X (n) \vert = n! K \vert X(0) \vert, $$
$$\forall 0 \leq i \leq n, m(U_i) = n! K \vert U_i \vert.$$ 
\end{proof}

Combining the above Corollary with the spectral descent result of Corollary \ref{Explicit spec descent coro}, yields Corollary \ref{Mixing + spec descent coro - general} that appeared in the introduction.

In the partite case, we have a similar mixing theorem, but the mixing depends only on the one-sided spectral gaps:

\begin{theorem}[Mixing Theorem - partite complexes]
\label{Mixing Theorem - partite complexes}
Let $X$ be as above and assume that $X$ is $(n+1)$-partite with sides $S_0,...,S_n$. Assume there is a constant $0< \lambda < 1$ such that for every $0 \leq k \leq n-1$ and every $\tau \in X(k-1)$, $\Spec ((M')_{\tau,0}^+) \subseteq [-1, \lambda] \cup \lbrace 1 \rbrace$, then for every sets $U_0 \subseteq S_0,...,U_n \subseteq S_n$,
$$\left\vert \dfrac{m(X(U_0,...,U_n))}{m(X(n))} - \dfrac{m(U_0) ... m(U_n)}{m(S_0) ... m(S_n)} \right\vert \leq C_n \lambda \min_{0 \leq i < j \leq n} \sqrt{\dfrac{m(U_i) m(U_j)}{m(S_i) m(S_j)}},$$
where 
$$C_n = \sum_{k=0}^{n-1} n! \dfrac{(n+1-k)^{n-k}}{(n-k-1)!}((k+1)(k+2)^{n-k} - (k+1)^{n-k+1} ).$$
\end{theorem}

\begin{proof}
The proof is similar to the proof in the non-partite case and therefore we will omit some details. We start by proving that 
\begin{dmath*}
m(X(U_0,...,U_n)) - \dfrac{(n+1)^{n-1}}{n!}\dfrac{m(U_0) m (U_1) ... m(U_n)}{m(X(0))^{n}}  =  
{\left\langle (d^*_{n-1} d_{n-1}) \chi_{X(U_0,...,U_{n-1})}, \chi_{X(U_1,...,U_{n})} \right\rangle - \dfrac{(n+1-k)^{n-k}}{(n-k)!} \left\langle \left( \prod_{i=1}^{n} P_{X(U_{i})} d_{-1} d_{-1}^* \right) \chi_{X(U_0)}, \chi_{X(U_n)} \right\rangle} = 
\sum_{k=0}^{n-1} \dfrac{(n+1-k)^{n-k}}{(n-k)!}\left\langle  \left(\left( \prod_{i=1}^{n-k} P_{X(U_{i},...,U_{k+i})} d_k^* d_k \right) - \left(\prod_{i=1}^{n-k} P_{X(U_{i},...,U_{k+i})} \left(\dfrac{n+1-k}{n-k} d_{k-1} d_{k-1}^* \right) \right) \right) \chi_{X(U_0,...,U_k)}, \chi_{X(U_{n-k},...,U_n)} \right\rangle. 
\end{dmath*}
We note that by Corollary \ref{weight of S_i coro}, for every $i$, $ m(S_i) = \frac{m(X(0))}{n+1}$, therefore
\begin{dmath*}
m(X(U_0,...,U_n)) - \dfrac{(n+1)^{n}}{n!}\dfrac{m(U_0) m (U_1) ... m(U_n)}{m(X(0))^{n}} = 
m(X(U_0,...,U_n)) - \dfrac{1}{n!}\dfrac{m(U_0) m (U_1) ... m(U_n)}{m(S_1)...m(S_n)}.
\end{dmath*}
Also note that by Proposition \ref{weight in n dim simplices}, $(n+1)! m(X(n)) = m(X(0)) = (n+1) m(S_i)$ for every $0 \leq i \leq n$, i.e., $n! m(X(n))= m(S_i)$. Therefore dividing the above equality by $m(X(n))$ yields that 
\begin{dmath*}
{\dfrac{1}{m(X(n))} \left( m(X(U_0,...,U_n)) - \dfrac{(n+1)^{n}}{n!}\dfrac{m(U_0) m (U_1) ... m(U_n)}{m(X(0))^{n}} \right)} = 
\dfrac{m(X(U_0,...,U_n))}{m(X(n))} - \dfrac{m(U_0) m (U_1) ... m(U_n)}{m(S_0) m(S_1)...m(S_n)}.
\end{dmath*}
Therefore
\begin{dmath*}
\left\vert \dfrac{m(X(U_0,...,U_n))}{m(X(n))} - \dfrac{m(U_0) m (U_1) ... m(U_n)}{m(S_0) m(S_1)...m(S_n)} \right\vert = 
\left\vert {\dfrac{1}{m(X(n))} \left( m(X(U_0,...,U_n)) - \dfrac{(n+1)^{n}}{n!}\dfrac{m(U_0) m (U_1) ... m(U_n)}{m(X(0))^{n}} \right)} \right\vert \leq 
\dfrac{1}{m(X(n))} \sum_{k=0}^{n-1} \dfrac{(n+1-k)^{n-k}}{(n-k)!} \\
\left\vert \left\langle  \left(\left( \prod_{i=1}^{n-k} P_{X(U_{i},...,U_{k+i})} d_k^* d_k \right) - \left(\prod_{i=1}^{n-k} P_{X(U_{i},...,U_{k+i})} \left(\dfrac{n+1-k}{n-k} d_{k-1} d_{k-1}^* \right) \right) \right) \chi_{X(U_0,...,U_k)},\\
 \chi_{X(U_{n-k},...,U_n)} \right\rangle \right\vert \leq^{\text{Corollary } \ref{product norm bound of indicator functions - partite case coro}} 
\dfrac{1}{m(X(n))}\lambda \sum_{k=0}^{n-1} \dfrac{(n+1-k)^{n-k}}{(n-k)!}  (n-k) ((k+1)(k+2)^{n-k} - (k+1)^{n-k+1} ) \sqrt{m(U_0) m(U_n)} = 
\lambda \sum_{k=0}^{n-1} n! \dfrac{(n+1-k)^{n-k}}{(n-k-1)!}   ((k+1)(k+2)^{n-k} - (k+1)^{n-k+1} ) \sqrt{\dfrac{m(U_0) m(U_n)}{m(S_0) m(S_n)}} = 
C_n \lambda \sqrt{\dfrac{m(U_0) m(U_n)}{m(S_0) m(S_n)}}.
\end{dmath*}
\end{proof}

As above, this Mixing Theorem can be simplified in the case where $X$ is partite-regular and the weight $m$ is the homogeneous weight - this simplified version is Theorem \ref{mixing thm for partite complexes - intro} that appeared in the introduction:
\begin{corollary}
\label{mixing for partite regular coro}
Let $X$ be a pure $n$-dimensional simplicial complex with the homogeneous weight function. Assume that $X$ is $(n+1)$-partite and partite-regular in the following sense: for every $0 \leq i \leq n$ there is a constant $K_i$ such that for every $\lbrace v \rbrace \in S_i$, $v$ is contained in exactly $K_i$ $n$-dimensional simplices of $X$, i.e., $m(\lbrace v \rbrace) = n! K_i$. Assume there is a constant $0< \lambda < 1$ such that for every $0 \leq k \leq n-1$ and every $\tau \in X(k-1)$, $\Spec ((M')_{\tau,0}^+) \subseteq [-1, \lambda] \cup \lbrace 1 \rbrace$, then for every sets $U_0 \subseteq S_0,...,U_n \subseteq S_n$,
$$\left\vert \dfrac{\vert X(U_0,...,U_n) \vert}{\vert X(n) \vert} - \dfrac{\vert U_0 \vert ... \vert U_n \vert}{\vert S_0 \vert ... \vert S_n \vert} \right\vert \leq C_n \lambda \min_{0 \leq i < j \leq n} \sqrt{\dfrac{\vert U_i \vert \vert U_j \vert}{\vert S_i \vert \vert S_j \vert}},$$
where 
$$C_n = \sum_{k=0}^{n-1} n! \dfrac{(n+1-k)^{n-k}}{(n-k-1)!}((k+1)(k+2)^{n-k} - (k+1)^{n-k+1} ).$$
\end{corollary}

\begin{proof}
The corollary follows from Theorem \ref{Mixing Theorem - partite complexes} by the following equalities:
$$m(X(U_0,...,U_n)) = \vert X(U_0,...,U_n) \vert, $$
$$\forall 0 \leq i \leq n, m(U_i) = n! K_i \vert U_i \vert, m(S_i)=n! K_i \vert S_i \vert.$$ 
\end{proof}

Combining the above Corollary with the spectral descent result of Corollary \ref{Explicit spec descent coro}, yields Corollary \ref{Mixing + spec descent coro - partite} that appeared in the introduction.

\section{Geometric overlapping}
\label{Geometric overlapping sec}
In \cite{Grom}, Gromov defined the geometric overlapping property for complexes. We'll define a weighted analogue of this property. We shall need the following definition first:
\begin{definition}
Let $X$ be an $n$-dimensional simplicial complex and let $f : X (0) \rightarrow \mathbb{R}^n$ be a map. The geometric extension of $f$ is the unique map $\widetilde{f} : X \rightarrow \mathbb{R}^n$ that extends $f$ affinely, i.e., for every $0 \leq k \leq n$ and every $\lbrace v_0,...,v_k \rbrace \in X (k)$, $\widetilde{f}$ maps $\lbrace v_0,...,v_k \rbrace$ to the simplex in $\mathbb{R}^n$ spanned by $f (v_0),...,f (v_k)$.
\end{definition}

Using the above definition, the geometrical overlapping property is defined as follows:
\begin{definition}
\label{geometric overlapping definition}
Let $X$ be a $n$-dimensional simplicial complex and let $\varepsilon >0$. We shall say that $X$ has $\varepsilon$-geometric overlapping if for every map $f :  X(0) \rightarrow \mathbb{R}^n$ and for the geometric  extension $\widetilde{f}$ of $f$, there is a point $x \in \mathbb{R}^n$ such that 
$$\vert \lbrace \sigma \in X(n) : x \in \widetilde{f} (\sigma)   \rbrace \vert \geq \varepsilon \vert  X (n) \vert .$$ 
In light of this definition, we denote
$$\overlap (X) = \min_{f: X(0) \rightarrow \mathbb{R}^n} \max_{x \in \mathbb{R}^n} \dfrac{\vert \lbrace \lbrace v_0,...,v_n \rbrace \in X(n) : x \in \conv \lbrace f(v_0),...,f(v_n) \rbrace \rbrace\vert}{\vert X(n) \vert},$$
and with this notation,  $X$ has $\varepsilon$-geometric overlapping if and only if $\overlap (X) \geq \varepsilon$. 
\end{definition}

As noted in \cite{Par}, a result by Pach can be used in order to derive geometric overlapping from mixing. Namely, Pach \cite{Pach} proved the following:
\begin{theorem}
For every $n$, there is a constant $\mathcal{P}_n$ such that for any sets $S_0,...,S_n$ of points in $\mathbb{R}^n$, there are subsets $U_0 \subseteq S_0,...,U_n \subseteq S_n$ with $\vert U_i \vert \geq \mathcal{P}_n \vert S_i \vert$ for every $0 \leq i \leq n$ and 
$$\bigcap_{x_0 \in U_0,...,x_n \in U_n} \conv \lbrace x_0,...,x_n \rbrace \neq \emptyset.$$  
\end{theorem}

Using our Corollaries \ref{mixing for regular coro}, \ref{mixing for partite regular coro} proven above, we can use this result to prove the following:

\begin{theorem}[Geometric overlapping - non-partite case]
Let $X$ be a pure $n$-dimensional simplicial complex with the homogeneous weight function. Assume that $X$ is $K$-regular in the following sense: for every $\lbrace v \rbrace \in X(0)$, $v$ is contained in exactly $K$ $n$-dimensional simplices of $X$, i.e., $m(\lbrace v \rbrace) = n! K$. Assume that there is a constant $0< \lambda < 1$ such that for every $0 \leq k \leq n-1$ and every $\tau \in X(k-1)$, $\Spec ((M')_{\tau,0}^+) \subseteq [-\lambda, \lambda] \cup \lbrace 1 \rbrace$, then 
$$\overlap (X) \geq n! \mathcal{P}_n \left( \left( \frac{\mathcal{P}_n}{n+1} \right)^n  - (n+1) C_n \lambda \right),$$
where $\mathcal{P}_n$ is the constant of Pach's theorem and
$$C_n = \sum_{k=0}^{n-1} ((k+1)(k+2)^{n-k} - (k+1)^{n-k+1} ).$$ 
\end{theorem}

\begin{theorem}[Geometric overlapping - partite case]
Let $X$ be a pure $n$-dimensional simplicial complex with the homogeneous weight function. Assume that $X$ is $(n+1)$-partite and partite-regular in the following sense: for every $0 \leq i \leq n$ there is a constant $K_i$ such that for every $\lbrace v \rbrace \in S_i$, $v$ is contained in exactly $K_i$ $n$-dimensional simplices of $X$, i.e., $m(\lbrace v \rbrace) = n! K_i$. Assume there is a constant $0< \lambda < 1$ such that for every $0 \leq k \leq n-1$ and every $\tau \in X(k-1)$, $\Spec ((M')_{\tau,0}^+) \subseteq [-1, \lambda] \cup \lbrace 1 \rbrace$, then
$$\overlap (X) \geq \mathcal{P}_n (\mathcal{P}_n^n-C_n \lambda),$$
where $\mathcal{P}_n$ is the constant of Pach's theorem and
$$C_n = \sum_{k=0}^{n-1} n! \dfrac{(n+1-k)^{n-k}}{(n-k-1)!}((k+1)(k+2)^{n-k} - (k+1)^{n-k+1} ).$$
\end{theorem}

We consider the partite case more natural in this setting and therefore, we will prove overlapping in this case and only sketch the proof of the non-partite case leaving the details to the reader. 

\begin{proof}
\textbf{Partite complex:} Let $X$ be a $(n+1)$-partite, partite-regular complex with the spectral gap in the links assumed above. Let $S_0,...,S_n$ be the sides of $X$ and let $f: X(0) \rightarrow \mathbb{R}^n$ be some map. By Pach's theorem there are $U_0 \subseteq S_0,...,U_n \subseteq S_n$ with $\vert U_n \vert \geq \mathcal{P}_n \vert S_n \vert$ and 
$$\bigcap_{v_0 \in U_0,...,v_n \in U_n} \conv \lbrace f(v_0),...,f(v_n) \rbrace \neq \emptyset.$$
Showing that 
$$\vert X(U_0,...,U_n) \vert \geq  \vert X(n) \vert \mathcal{P}_n (\mathcal{P}_n^n - C_n \lambda ),$$
will complete the proof. Indeed, by Corollary \ref{mixing for partite regular coro}, 
$$\left\vert \dfrac{\vert X(U_0,...,U_n) \vert}{\vert X(n) \vert} - \dfrac{\vert U_0 \vert ... \vert U_n \vert}{\vert S_0 \vert ... \vert S_n \vert} \right\vert \leq C_n \lambda \min_{0 \leq i < j \leq n} \sqrt{\dfrac{\vert U_i \vert \vert U_j \vert}{\vert S_i \vert \vert S_j \vert}},$$
and therefore
\begin{dmath*}
\vert X(U_0,...,U_n) \vert \geq \vert X(n) \vert \left( \dfrac{\vert U_0 \vert ... \vert U_n \vert}{\vert S_0 \vert ... \vert S_n \vert} - C_n \lambda \min_{0 \leq i < j \leq n} \sqrt{\dfrac{\vert U_i \vert \vert U_j \vert}{\vert S_i \vert \vert S_j \vert}} \right) \geq \\
 \vert X(n) \vert \sqrt{\dfrac{\vert U_0 \vert \vert U_1 \vert}{\vert S_0 \vert \vert S_1 \vert}} \left( \sqrt{\dfrac{\vert U_0 \vert \vert U_1 \vert}{\vert S_0 \vert \vert S_1 \vert}} \dfrac{\vert U_2 \vert ... \vert U_n \vert}{\vert S_2 \vert ... \vert S_n \vert} - C_n \lambda \right) \geq 
 \vert X(n) \vert \mathcal{P}_n (\mathcal{P}_n^n - C_n \lambda ).
\end{dmath*}

\textbf{Non-partite complex:} First, arbitrarily divide $X(0)$ into $(n+1)$ disjoint sets $S_0,...,S_n$ of equal size and then repeat the same proof using Corollary \ref{mixing for regular coro} instead of Corollary \ref{mixing for partite regular coro}.
\end{proof}

Also, combining the above Theorems with the spectral descent result of Corollary \ref{Explicit spec descent coro}, yields the following:
\begin{corollary}
Let $X$ be a pure $n$-dimensional simplicial complex with the homogeneous weight function. Assume that $X$ is $K$-regular in the following sense: for every $\lbrace v \rbrace \in X(0)$, $v$ is contained in exactly $K$ $n$-dimensional simplices of $X$, i.e., $m(\lbrace v \rbrace) = n! K$. Assume that there is a constant $0< \lambda < 1$ such that for every $\tau \in X(n-2)$, $\Spec ((M')_{\tau,0}^+) \subseteq [-\frac{\lambda}{1+(n-1)\lambda}, \frac{\lambda}{1+(n-1)\lambda}] \cup \lbrace 1 \rbrace$, then 
$$\overlap (X) \geq n! \mathcal{P}_n \left( \left( \frac{\mathcal{P}_n}{n+1} \right)^n  - (n+1) C_n \lambda \right),$$
where $\mathcal{P}_n$ is the constant of Pach's theorem and
$$C_n = \sum_{k=0}^{n-1} ((k+1)(k+2)^{n-k} - (k+1)^{n-k+1} ).$$ 
\end{corollary}

\begin{corollary}
Let $X$ be a pure $n$-dimensional simplicial complex with the homogeneous weight function. Assume that $X$ is $(n+1)$-partite and partite-regular in the following sense: for every $0 \leq i \leq n$ there is a constant $K_i$ such that for every $\lbrace v \rbrace \in S_i$, $v$ is contained in exactly $K_i$ $n$-dimensional simplices of $X$, i.e., $m(\lbrace v \rbrace) = n! K_i$. Assume there is a constant $0< \lambda < 1$ such that every $\tau \in X(n-2)$, $\Spec ((M')_{\tau,0}^+) \subseteq [-1,  \frac{\lambda}{1+(n-1)\lambda}] \cup \lbrace 1 \rbrace$, then
$$\overlap (X) \geq \mathcal{P}_n (\mathcal{P}_n^n-C_n \lambda),$$
where $\mathcal{P}_n$ is the constant of Pach's theorem and
$$C_n = \sum_{k=0}^{n-1} n! \dfrac{(n+1-k)^{n-k}}{(n-k-1)!}((k+1)(k+2)^{n-k} - (k+1)^{n-k+1} ).$$
\end{corollary}

We note that the partite case proves that quotients of Affine buildings of large enough thickness have geometric overlapping and along as the quotient preserves the $1$-dimensional links of the builidng, thus generalizing the result of Fox, Gromov, Lafforgue, Naor and Pach \cite{FGLNP} who showed geometric overlapping for quotients of $\widetilde{A}$-buildings. Also, the new examples of high dimensional expanders constructed by Kaufman and the author \cite{KO-construction} fulfill the assumption of the partite mixing (and the spectral gaps can be chosen to be arbitrarily small) and therefore provide new examples of complexes with the geometric overlapping property. 

\appendix

\section{Spectral descent for random walks}

In \cite{LocSpecI}, we proved spectral descent result - we showed that given a simplicial complex with connected links, one can derive bounds on the spectra of the Laplacians in links of every dimension, based on bounds on the spectrum of the Laplacians in the $1$-dimensional links. In the paper above, we replaced the Laplacians with Random walk operators. 
 
In this appendix, we state the spectral descent results of \cite{LocSpecI} in the language of random walks. These results are easy to derive, because the graph Laplacian $\Delta_0^+$ is basically defined by $\Delta_0^+ = I-(M')_0^+$, and therefore we do not include the proof. We also use this result to find a sufficient condition on the spectral gaps of $1$-dimensional links, such that the conditions of our mixing theorems are fulfilled.  

\begin{theorem}[Spectral descent for random walks]
\label{spec descent thm}
Let $X$ be a weighted pure $n$-dimensional simplicial complex, such that all the links of $X$ of dimension $\geq 1$ (including $X$ itself) are connected. For $0 \leq k \leq n-1$, let $0 \leq \mu_k \leq 1,-1 \leq \nu_k \leq 0$ be constants such that for every $\sigma \in X(k-1)$, $\Spec ((M_\sigma')_0^{+}) \subseteq [\nu_k, \mu_k] \cup \lbrace 1 \rbrace$.  Then for every $0 \leq k \leq n-2$, 
$$\mu_k \leq \dfrac{\mu_{k+1}}{1-\mu_{k+1}},$$
$$\nu_k \geq \dfrac{\nu_{k+1}}{1-\nu_{k+1}}.$$
\end{theorem}

A simple induction leads to the following:
\begin{corollary}
\label{spec gap descent induction coro}
Let $X$ be a weighted pure $n$-dimensional simplicial complex, such that all the links of $X$ of dimension $\geq 1$ (including $X$ itself) are connected, then for every $0 \leq k \leq n-2$,
$$\mu_k \leq \dfrac{\mu_{n-1}}{1-(n-1-k)\mu_{n-1}},$$
$$\nu_k \geq \dfrac{\nu_{n-1}}{1-(n-1-k)\nu_{n-1}}.$$
\end{corollary}

A corollary of the above corollary is the following:
\begin{corollary}
\label{Explicit spec descent coro}
Let $X$ be a weighted pure $n$-dimensional simplicial complex, such that all the links of $X$ of dimension $\geq 1$ (including $X$ itself) are connected, then and let $0 < \lambda \leq 1$ be some constant. If $\mu_{n-1} \leq \frac{\lambda}{1+(n-1)\lambda}$, then for every $0 \leq k \leq n-1$, $\mu_k \leq \lambda$, i.e., for every $\sigma \in \bigcup_{k=-1}^{n-2} X(k)$, $\Spec ((M_\sigma')_0^{+}) \subseteq [-1, \lambda] \cup \lbrace 1 \rbrace$.

Moreover, if $\mu_{n-1} \leq \frac{\lambda}{1+(n-1)\lambda}$ and $\frac{-\lambda}{1+(n-1)\lambda} \leq \nu_{n-1}$, then for every $0 \leq k \leq n-1$,  $\mu_k \leq \lambda$ and $-\lambda \leq \nu_k$, i.e.,
for every $\sigma \in \bigcup_{k=-1}^{n-2} X(k)$, $\Spec ((M_\sigma')_0^{+}) \subseteq [-\lambda, \lambda] \cup \lbrace 1 \rbrace$.
\end{corollary}

\begin{proof}
By the above corollary, if $\mu_{n-1} \leq \frac{\lambda}{1+(n-1)\lambda}$ then for every $0 \leq k \leq n-2$ we have that
\begin{dmath*}
\mu_k \leq \dfrac{\mu_{n-1}}{1-(n-1-k)\mu_{n-1}} \leq \dfrac{\frac{\lambda}{1+(n-1)\lambda}}{1-(n-1-k)\frac{\lambda}{1+(n-1)\lambda}} \leq \dfrac{\frac{\lambda}{1+(n-1)\lambda}}{1-(n-1)\frac{\lambda}{1+(n-1)\lambda}} = \lambda,
\end{dmath*}
The proof of the second assertion is similar.
\end{proof}

\bibliographystyle{alpha}
\bibliography{bibl}
\end{document}